\documentclass[12pt]{elsarticle}
\usepackage{lineno,hyperref}
\modulolinenumbers[5]
\usepackage{amsmath, amssymb, amsthm, empheq}
\usepackage{tikz}
\usepackage{graphicx,subfigure}
\usepackage{nameref}
\usetikzlibrary{arrows,shapes,chains}
\usepackage[a4paper, left=2.5cm, right=2.5cm, top=3cm, bottom=3cm]{geometry}

\newtheorem{thm}{Theorem}[section]

\newtheorem{lem}[thm]{Lemma}
\newtheorem{pro}[thm]{Proposition}

\newtheorem*{thmA}{Theorem A}
\newtheorem*{thmB}{Theorem B}
\newtheorem*{thmC}{Theorem C}
\newtheorem*{thmD}{Theorem D}
\newtheorem*{thmE}{Theorem E}
\newtheorem*{thmF}{Theorem F}
\newtheorem*{thmG}{Theorem G}

\newtheorem*{BC}{Base case}
\newtheorem*{BF}{Basic fact}

\newtheorem*{FactA}{Fact A}
\newtheorem*{FactB}{Fact B}
\newtheorem*{ClaimA}{Claim 1}
\newtheorem*{ClaimB}{Claim 2}
\newtheorem*{ClaimC}{Claim 3}
\newtheorem*{ClaimD}{Claim 4}

\newtheorem*{conj}{Conjecture}

\numberwithin{equation}{section}
\journal{Journal of \LaTeX\ Templates}

\makeatletter
\newcommand{\lefttag}[2]{%
  \def\@currentlabel{#1}%
  \begin{flalign}
    \text{(#1)}\quad & & #2 && \nonumber
  \end{flalign}
}
\makeatother

\begin{document}

\begin{frontmatter}

\title{A proof of the Freiman-Lev conjecture\tnoteref{mytitlenote}}
\tnotetext[mytitlenote]{This work was supported by the National Natural Science Foundation of China (Grant
Nos. 12101007 and 12371003) and the Natural Science Foundation of Anhui Province (Grant No. 2008085QA06).}

\author[mymainaddress]{Yujie Wang}
\ead{wangyujie9291@126.com}

\author[mymainaddress]{Min Tang\corref{mycorrespondingauthor}}
\ead{tmzzz2000@163.com}
\cortext[mycorrespondingauthor]{Corresponding author}

\address[mymainaddress]{School of Mathematics and Statistics, Anhui Normal University, Wuhu 241002, P. R. China}

\begin{abstract} Let $A=\{a_{0}, a_{1}, \ldots, a_{k-1}\}$ be a set of $k>7$ integers such that $0=a_{0}<a_1<\cdots<a_{k-1}$ and $\gcd(A)=1$.
 The set $2^{\wedge}A=\{a+b: a, b\in A, a\neq b\}$ is called the restricted sumset of $A$.
 The Freiman-Lev conjecture is a well-known conjecture concerning restricted sumsets [V.F. Lev, Restricted set addition in groups, I. The classical setting, J. London Math. Soc. 62(2000), 27-40].
 Up to now, Freiman-Lev conjecture is still open for all $a_{k-2}\geqslant 2k-4$ and $a_{k-1}\geqslant 2k-2$.
 In this paper, we complete the proof of the Freiman-Lev conjecture by resolving this final and most challenging case.

\end{abstract}


\end{frontmatter}


\section{Introduction}
Let $A$ be a set of integers and let $\gcd(A)$ be the greatest common divisor of all nonzero elements of $A$. We define the sumsets of $A$ and the restricted sumsets of $A$ to be
$$2A=\{a+b: a, b\in A\}, \ \ 2^{\wedge}A=\{a+b: a, b\in A, a\neq b\}, $$
respectively.
Write $[a, b]=\{x\in \mathbb{Z}\ |\ a\leqslant x\leqslant b\}$.

In 1959, G.A. Freiman \cite{Freiman} obtained the famous Freiman's $2A$ theorem.
\begin{thmA}
Let $A\subseteq [0, l]$ be a set of $k\geqslant 3$ integers such that $0, l\in A$, $\gcd(A)=1$. Then
$$|2A|\geqslant \left\{\begin{array}{ll}
l+k &\text{ if } l\leqslant 2k-3, \\
3k-3 &\text{ if } l\geqslant 2k-2. \end{array}\right.$$
\end{thmA}
Freiman's $2A$ theorem has continued to attract significant attention in number theory and combinatorics,
the reader can see \cite{Freiman64, Freiman1, Jin2007}.

It is natural to study the lower bound for the cardinality of the restricted sumset $2^{\wedge}A$.
However, the problems of the restricted sumsets demand analyzing and handling of all possible subtle structures which makes the problems far more complex than dealing with general sumsets.
Early work on restricted sumsets can be found in, e.g., Nathanson's inverse theorems for subset sums \cite{Nathanson1}.
Nevertheless, it was not until 1999 that G. A. Freiman, L. Low and J. Pitman \cite{Freiman1999} obtained the following result, which builds upon his classical $2A$ theorem.
\begin{thmB}
Let $A\subseteq [0, l]$ be a set of $k\geqslant 3$ integers such that $0, l\in A$, $\gcd(A)=1$. Then
$$|2^{\wedge}A|\geqslant \left\{\begin{array}{ll}
0.5(l+k)+k-3.5 &\text{ if } l\leqslant 2k-3, \\
2.5k-5 &\text{ if } l\geqslant 2k-2. \end{array}\right.$$
\end{thmB}

In 2000, V.F. Lev \cite{Lev2} improved the lower bound for the cardinality of $2^{\wedge}A$, and remarked that the following conjecture (in personal communication with Freiman).
\begin{thmC}
Let $A\subseteq [0, l]$ be a set of $k\geqslant 3$ integers such that $0, l\in A$ and $\gcd(A)=1$. Then
$$|2^{\wedge}A|\geqslant \left\{\begin{array}{ll}
l+k-2 &\text{ if } l\leqslant 2k-5, \\
(\theta+1)k-6 &\text{ if } l\geqslant 2k-4, \end{array}\right.$$
where $\theta=(1+\sqrt{5})/2$.
\end{thmC}

\begin{conj} Let $A\subseteq [0, l]$ be a set of $k>7$ integers such that $0, l\in A$ and $\gcd(A)=1$. Then
$$|2^{\wedge}A|\geqslant \left\{\begin{array}{ll}
l+k-2 &\text{ if } l\leqslant 2k-5, \\
3k-7 &\text{ if } l\geqslant 2k-4. \end{array}\right.$$
\end{conj}

In 2002, Schoen \cite{Sch} almost solved the conjecture of Freiman and Lev.

\begin{thmD}
Let $A$ be a set of $k>7$ integers such that $\gcd(A)=1$, $A\subseteq [0,l]$, and $0,l\in A$. Then
\begin{displaymath}
|2^{\wedge}A|\geq
\begin{cases}
l+k-2, &\text{ if } l\leq2k-5,\\
3k+o(k), &\text{ if } l\geq2k-4.
\end{cases}
\end{displaymath}
\end{thmD}

In 2025, we \cite{Wang} introduced the gap set $W$ defined as integers in $[0, l]\backslash A$ whose sum with either $0$ or $l$ is absent from the restricted sumset $2^{\wedge}A$. we proved that if $l\leqslant2k-3$, then $|W|\leqslant 2$. In particular, when $|W|=2$, the paper provided a complete characterization of the structure of the original set $A$. Building on this analysis, we successfully proved the Freiman-Lev conjecture for the critical diameter range $2k-4\leqslant l\leqslant 2k-3$.
\begin{thmE}\label{thmE}
Let $A\subseteq [0, l]$ be a set of $k\geqslant 3$ integers such that $0, l\in A$ and $\gcd(A)=1$.
If $2k-4\leqslant l\leqslant 2k-3$, then $|2^{\wedge}A|\geqslant 3k-7$.
\end{thmE}
Recently, we \cite{Wang2025} develop new combinatorial and number-theoretic lemmas to handle the intricate case where $a_{i}<2i$ for each integer $1\leqslant i\leqslant k-2$ (we call this type set the locally dense set). We show that the Freiman-Lev conjecture is true for locally dense set. Moreover, extremal sets are highly regular-unions of arithmetic progressions with a common difference modulo $3$.
\begin{thmF}
Let $A=\{a_{0}, a_{1}, \ldots, a_{k-1}\}$ be a set of $k\geqslant 3$ integers such that $0=a_{0}<a_{1}<\cdots<a_{k-1}$, $a_{i}<2i$ for all $i=1, \ldots, k-2$, $a_{k-1}\geqslant 2k-2$ and $\gcd(A)=1$.
Then
$$\left|2^{\wedge}A\right|\geqslant 3k-7. $$
Moreover, $|2^{\wedge}A|=3k-7$ if and only if $k\geqslant 6$ and one of the following cases holds:

(1) $k\equiv 0\pmod{3}$ and
$$A=\left([0,k-3]\cap 3\mathbb{Z}\right)\cup \left([1,2k-2]\cap (3\mathbb{Z}+1)\right);$$

(2) $k\equiv 1\pmod{3}$ and
$$A=\left([0,2k-2]\cap 3\mathbb{Z}\right)\cup \left([1,k-3]\cap (3\mathbb{Z}+1)\right).$$
\end{thmF}

Moreover, in the critical case where $a_{k-1}=2k-3$, we provided a complete classification of all extremal sets attaining the lower bound $|2^{\wedge}A|=3k-7$ (see Theorem 1.3 of \cite{Wang2025}, for the exhaustive list). Combining these structural insights with inductive arguments, we obtain the following result.

\begin{thmG}
Let $A=\{a_{0}, a_{1}, \ldots, a_{k-1}\}$ be a set of $k\geqslant 3$ integers such that $0=a_{0}<a_{1}<\cdots<a_{k-1}$ and $\gcd(A)=1$.
If $a_{k-2}<2k-4$, $a_{k-1}\geqslant 2k-2$, then $|2^{\wedge}A|\geqslant 3k-7$.
\end{thmG}

Up to now, the Freiman-Lev conjecture has remained an open unsolved problem when $a_{k-2}\geqslant 2k-4$ and $a_{k-1}\geqslant 2k-2$. In this paper, we resolve this final and most challenging case.

\begin{thm}\label{T} Let $A=\{a_{0}, a_{1}, \ldots, a_{k-1}\}$ be a set of $k>7$ integers such that
$$0=a_{0}<a_{1}<\cdots<a_{k-1},\; \gcd(A)=1,\;a_{k-2}\geqslant 2k-4,\;a_{k-1}\geqslant 2k-2.$$
Then $|2^{\wedge}A|\geqslant 3k-7$.
\end{thm}
By Theorems C, E, G and Theorem \ref{T}, we know that Freiman-Lev conjecture is true.

Throughout this paper, let $A=\{a_{0}, a_{1}, \ldots, a_{k-1}\}$ such that
$$0=a_{0}<a_{1}<\cdots<a_{k-1} \text{ and } \gcd(A)=1. $$
Let $S(w)=\{w, w+a_{k-1}\}$ for any integer $w$ and
$$W=\left\{w\in [0, a_{k-1}]\backslash A:  S(w)\cap 2^{\wedge}A=\emptyset\right\}. $$

The paper is organized as follows. In Section \ref{S2}, we establish some fundamental combinatorial lemmas. We provide precise descriptions of how elements of the set $A$ must distribute within certain intervals when $W$ is non-empty, especially under the critical conditions where the largest elements $a_{k-1}=2k-3$ or $a_{k-1}=2k-4$. These results form the basic technical toolkit for the detailed structural analysis in subsequent sections.

In section \ref{S4}, we focus on the class of locally dense set. We establish that certain specific initial configurations (for example $(a_2,a_3,a_4)=(3,5,6))$ force the restricted sumset size of the set to satisfy a strengthened lower bound, typically $3k-5$.

In section \ref{S5}, we establishe a series of structural classifications, and derived bounds for integer sets whose largest element is near the critical thresholds (mainly $a_{k-1}=2k-3$ or $a_{k-1}=2k-4$). We obtain some inverse results (complete structural characterizations of extremal sets).

In section \ref{S7}, we investigate the growth of the restricted sumset when a dense set is extended by appending some elements that are larger than its original maximum.

In section \ref{S3}, we focus on the translative and recursive analysis of internal additive generation. We demonstrate how structural conditions such as a large greatest common divisor ($\gcd$) for a translated set or specific additive generation patterns force stronger lower bounds for $|2^{\wedge}A|$.

In section \ref{S6}, we first conduct a complete structural classification (inverse results) of sets at the critical values $a_{k-1}=2k-3$ or $2k-4$ and study their stability after adding large elements; subsequently, we employ translation operations (such as $B=A\backslash \{a_0\}-a_1$) and recursive generation relations to transform complex internal additive conditions into tractable subproblems. Finally, through a carefully designed inductive framework that integrates the various lemmas and propositions from Sections \ref{S2}-\ref{S3}, the proof is completed, thereby resolving the final and most challenging case of the Freiman-Lev conjecture.

\section{Lemmas}\label{S2}
The following two results are immediate consequence of pigeonhole principle.

\begin{lem}\label{LL3-2}
Let $A$ be a finite set of $k\geqslant 3$ integers such that $2k-4\leqslant a_{k-1}\leqslant 2k-3$.
Assume that $W\neq \emptyset$. Then for any $w\in W$, we have
\begin{equation}\label{e1-2}
|\{i, w-i\}\cap A|=1, \ \ i=0, 1, \ldots, \left\lfloor\frac{w}{2}\right\rfloor,
\end{equation}
\begin{equation}\label{e1-3}
|\{i, w+a_{k-1}-i\}\cap A|=1, \ \ i=w, w+1, \ldots, \left\lfloor\frac{w+a_{k-1}}{2}\right\rfloor.
\end{equation}
In particular, if $a_{k-1}=2k-4$, then $w$ is even for any $w\in W$.
\end{lem}

\begin{lem}\label{LL3-8}
Let $A$ be a finite set of $k\geqslant 4$ integers such that $a_{k-1}=2k-3$.
Assume that $W\neq \emptyset$. Then for any $w\in W$, we have

(1) If $w+1\notin 2^{\wedge}A$ and $2\mid w$, then
$$[0, w]\cap A=\left[0, \frac{w}{2}\right]. $$

(2) If $w+1\notin 2^{\wedge}A$ and $2\nmid w$, then
$$[0, w]\cap A=\left[0, \frac{w-1}{2}\right] $$
or there exists an integer $0\leqslant a<\frac{w-1}{2}$ such that
$$[0, w]\cap A=[0, a]\cup \left[\frac{w+1}{2}, w-a-1\right]. $$

(3) If $w-1\notin 2^{\wedge}A$, then $2\nmid w$ and
$$[0, w]\cap A=\left[0, \frac{w-1}{2}\right]. $$
\end{lem}

\begin{lem}\label{LL3-6}
Let $A$ be a finite set of $k\geqslant 4$ integers such that $a_{k-1}=2k-4$.
Assume that $W\neq \emptyset$. Then for any $w\in W$, we have

(1) $|S(w-1)\cap 2^{\wedge}A|\geqslant 1$. Moreover, if $w-1\notin A$, then
$$|S(w-1)\cap 2^{\wedge}A|=1\Longleftrightarrow\left[k-2+\frac{w}{2}, a_{k-1}\right]\subseteq A. $$

(2) $|S(w+1)\cap 2^{\wedge}A|\geqslant 1$. Moreover, if $w+1\notin A$, then
$$|S(w+1)\cap 2^{\wedge}A|=1\Longleftrightarrow\left[0, \frac{w}{2}\right]\subseteq A. $$
\end{lem}

\begin{proof}
Since $a_{k-1}=2k-4$, by Lemma \ref{LL3-2} we have $2\mid w$ and
$$\frac{w}{2}, \ \ k-2+\frac{w}{2}\in A. $$

(1) If $w-1\in A$, then $w-1\in 2^{\wedge}A$. If $w-1\not\in A$, then $w-1\in 2^{\wedge}A$. Otherwise,
\begin{equation}\label{e1-10}|\{i, w-1-i\}\cap A|\leqslant 1, \ \ i=0, 1, \ldots, \frac{w}{2}-1. \end{equation}
Thus,
$$
\frac{w}{2}\in A \overset{(\ref{e1-10})}\Longrightarrow \frac{w}{2}-1\not\in A \overset{(\ref{e1-2})}\Longrightarrow \frac{w}{2}+1\in A\overset{(\ref{e1-10})}\Longrightarrow\frac{w}{2}-2\not\in A\overset{(\ref{e1-2})}\Longrightarrow\cdots\overset{(\ref{e1-2})}\Longrightarrow w-1\in A,
$$
a contradiction. Hence, $|S(w-1)\cap 2^{\wedge}A|\geqslant 1$.

If $|S(w-1)\cap 2^{\wedge}A|=1$, then $w-1+a_{k-1}\not\in 2^{\wedge}A$. Thus,
\begin{equation}\label{e1-26}|\{i, w-1+a_{k-1}-i\}\cap A|\leqslant 1, \ \ i=w-1, \ldots, k-3+\frac{w}{2}. \end{equation}
Hence,
$$
k-2+\frac{w}{2}\in A \overset{(\ref{e1-26})}\Longrightarrow k-3+\frac{w}{2}\not\in A \overset{(\ref{e1-3})}\Longrightarrow k-1+\frac{w}{2}\in A
\overset{(\ref{e1-26})}\Longrightarrow\cdots \overset{(\ref{e1-3})}\Longrightarrow a_{k-1}\in A,
$$
it follows that
$$\left[k-2+\frac{w}{2}, a_{k-1}\right]\subseteq A. $$

Conversely, if $\left[k-2+\frac{w}{2}, a_{k-1}\right]\subseteq A$, then by (\ref{e1-2}) we have
$$\left[w, k-3+\frac{w}{2}\right]\cap A=\emptyset, $$
thus, $w-1+a_{k-1}\not\in 2^{\wedge}A$, it follows that $|S(w-1)\cap 2^{\wedge}A|=1$.

(2) Similarly, we can show that $w+1+a_{k-1}\in 2^{\wedge}A$. The remainder proof is omitted here.

This completes the proof of Lemma \ref{LL3-6}.
\end{proof}

\begin{lem}[\cite{Wang}, Theorem 1.1]\label{LL3-1}
Let $A$ be a set of $k\geqslant 3$ integers such that $a_{k-1}\leqslant 2k-3$. Then $|W|\leqslant 2$.
\end{lem}

\begin{lem}\label{LL3-7}
Let $A$ be a finite set of $k\geqslant 3$ integers such that $a_{k-1}=2k-3$ and $|2^{\wedge}A|=3k-6$.
Then $|W|=1$ or $2$. Moreover,

(1) $|W|=1$ if and only if
$$|S(b)\cap 2^{\wedge}A|=1 \text{ for all } b\in [0, a_{k-1}]\backslash(A\cup W). $$

(2) $|W|=2$ if and only if there exists a unique integer $\xi\in  [0, a_{k-1}]\backslash(A\cup W)$ such that
$$|S(\xi)\cap 2^{\wedge}A|=2. $$
\end{lem}

\begin{proof}
Write
$$T=\{a_{i}: 1\leqslant i\leqslant k-2\}\cup\{a_{i}+a_{k-1}: 0\leqslant i\leqslant k-2\}. $$
Then $T\subseteq 2^{\wedge}A$ and $|T|=2k-3$. By the definition of $W$, we have
$$|S(b)\cap 2^{\wedge}A|\geqslant 1, \ \ b\in [0, a_{k-1}]\backslash(A\cup W), $$
thus,
$$|2^{\wedge}A|\geqslant |T|+|[0, a_{k-1}]\backslash(A\cup W)|=3k-5-|W|. $$
By Lemma \ref{LL3-1} we have $|W|\leqslant 2$, combining with $|2^{\wedge}A|=3k-6$, we have $|W|=1$ or $2$.

If $|W|=1$, then for any $b\in  [0, a_{k-1}]\backslash(A\cup W)$, we have $|S(b)\cap 2^{\wedge}A|=1$.
Otherwise, if there exists an integer $b\in  [0, a_{k-1}]\backslash(A\cup W)$ such that $|S(b)\cap 2^{\wedge}A|=2$,
then
$$|2^{\wedge}A|\geqslant |T|+|[0, a_{k-1}]\backslash(A\cup W)|+1=3k-5. $$
a contradiction.

If $|W|=2$, then there exists a unique integer $\xi\in  [0, a_{k-1}]\backslash(A\cup W)$ such that $|S(\xi)\cap 2^{\wedge}A|=2$.
Otherwise,
$$|2^{\wedge}A|\geqslant |T|+|[0, a_{k-1}]\backslash(A\cup W)|+2=3k-5. $$
a contradiction.

Since $|W|=1$ or $2$, combine with the above discussion, the sufficiency is easily to obtain by contradiction.

This completes the proof of Lemma \ref{LL3-7}.
\end{proof}

\begin{lem}\label{L4}
Let $A$ be a set of $k\geqslant 9$ integers such that $a_{k-3}\leqslant 2k-7$, $a_{k-2}=2k-4$, $a_{k-1}=2k-3$.
If $a_{k-3}-2\not\in A$, $a_{k-3}-1\not\in A$ and $|2^{\wedge}A|=3k-6$, then $a_{k-3}-1\notin 2^{\wedge}A$.
\end{lem}

\begin{proof}
Assume that $a_{k-3}-1\in 2^{\wedge}A$. Noting that
$$a_{k-3}-1+a_{k-1}=a_{k-3}+a_{k-2}\in 2^{\wedge}A, $$
we have
$$|S(a_{k-3}-1)\cap 2^{\wedge}A|=2. $$
By Lemma \ref{LL3-7} (2) we have $|W|=2$ and
\begin{equation}\label{eL4-1}|S(b)\cap 2^{\wedge}A|=1, \ \ b\in [0, a_{k-1}]\backslash(A\cup W), b\neq a_{k-3}-1.\end{equation}

Write $W=\{w_{1}, w_{2}\}$ with $w_{1}<w_{2}$. If $w_{1}>a_{k-3}$, then by (\ref{e1-3}) and $a_{k-2}=2k-4$ we have
$$2=\left|\left[w_{1}, a_{k-1}\right]\cap A\right|=\left\lfloor\frac{w_{1}+a_{k-1}}{2}\right\rfloor-w_{1}+1, $$
thus, $w_{1}=2k-6$, $w_{2}=2k-5$. By Lemma \ref{LL3-8} (1) we have $[0, k-3]\subseteq A$. Thus,
$$A=[0, k-3]\cup \{2k-4, 2k-3\}, $$
we know that $a_{k-3}=k-3$ and $a_{k-3}-1\in A$, a contradiction. Hence, $w_{1}<a_{k-3}$.

Write $w=\max(W\cap [0, a_{k-3}])$. Now, we shall show that $w=w_{1}$.
Clearly, $w\neq a_{k-3}-1$. If $w=a_{k-3}-2$, then by (\ref{e1-3}) we have
$$3=|\{a_{k-3}, a_{k-2}, a_{k-1}\}|=|[w, a_{k-1}]\cap A|=\left\lfloor\frac{w+a_{k-1}}{2}\right\rfloor-w+1, $$
thus, $a_{k-3}\geqslant  2k-6$, a contradiction. Hence, $w\leqslant a_{k-3}-3$.

By $a_{k-2}=2k-4$, $a_{k-3}\leqslant 2k-7$ and (\ref{e1-2}), (\ref{e1-3}), we have
$$w+1\not\in A, \ \ a_{\left\lfloor\frac{w}{2}\right\rfloor+1}=w+2, $$
thus,
$$w+1+a_{k-1}=w+2+a_{k-2}\in 2^{\wedge}A. $$
By (\ref{eL4-1}) we have
$$w+1\not\in 2^{\wedge}A. $$

If $2\mid w$, then by Lemma \ref{LL3-8} (1) we have
\begin{equation}\label{lem2.6-1}[0, w]\cap A=\left[0, \frac{w}{2}\right], \end{equation}
thus, $[1, w-1]\subseteq 2^{\wedge}A$, we have $w=w_{1}$.

If $2\nmid w$ and $1\in A$, then $w-1\not\in A$. If $w-1+a_{k-1}\not\in 2^{\wedge}A$, then
$$|\{i, w-1+a_{k-1}-i\}\cap A|\leqslant 1, \ \ i=w-1, \ldots, k-2+\frac{w-1}{2}, $$
combining with $w+1, 2k-5\not\in A$, we have
$$k-\frac{w+1}{2}=|[w-1, a_{k-1}]\cap A|\leqslant k-\frac{w+3}{2}, $$
a contradiction. Thus,
$$w-1+a_{k-1}\in 2^{\wedge}A. $$
By (\ref{eL4-1}) we have $w-1\notin 2^{\wedge}A$.
By Lemma \ref{LL3-8} (3) we have
\begin{equation}\label{lem2.6-2}[0,w]\cap A=\left[0, \frac{w-1}{2}\right], \end{equation}
thus, $[1, w-2]\subseteq 2^{\wedge}A$, we also have $w=w_{1}$.

Assume that $2\nmid w$ and $1\notin A$. If $w=1$, then $w=w_{1}$.
If $w>1$, then by Lemma \ref{LL3-8} (2) we have
\begin{equation}\label{lem2.6-3}[0, w]\cap A=\{0\}\cup \left[\frac{w+1}{2}, w-1\right]. \end{equation}
If $w=w_{2}$, then
$$1\leqslant w_{1}\leqslant \frac{w-1}{2}. $$
By (\ref{e1-2}) we have
$$|[0, w_{1}]\cap A|=\frac{w_{1}+1}{2}\geqslant 1, $$
combining with (\ref{lem2.6-3}) we have $w_{1}=1$.
By $a_{k-2}=2k-4$ and $a_{k-3}\leqslant 2k-7$, applying (\ref{e1-3}) to $w_{1}$ we have
$$2\notin A, \ \ 3, 4\in A, $$
combining with (\ref{lem2.6-3}) we have $w_{2}=5$. Again apply (\ref{e1-3}) to $w_{1}$ we have $a_{k-3}=2k-7$.
By $a_{k-2}=2k-4$, applying (\ref{e1-3}) to $w_{2}$ we have $6\notin A$. Again apply (\ref{e1-3}) to $w_{1}$ we have $a_{k-3}-1\in A$,
a contradiction. Thus, $w=w_{1}$.

Hence, $w_{2}>a_{k-3}$, so,
$$w_{2}=2k-6 \text{ or } 2k-5. $$

If $2\mid w$, then by (\ref{e1-2}) and (\ref{lem2.6-1}) we have
$$a_{k-3}=w_{2}-\frac{w}{2}-1 \text{ and } a_{k-3}-1\in A, $$
a contradiction.

If $2\nmid w$ and $1\in A$, then by (\ref{e1-2}) and (\ref{lem2.6-2}) we have
$$a_{k-3}=w_{2}-\frac{w-1}{2} \text{ and } a_{k-3}-1\in A, $$
a contradiction.

Assume that $2\nmid w$ and $1\notin A$. If $w=1$, then by $a_{k-2}=2k-4$ and (\ref{e1-3}) we have $2\notin A$.
By (\ref{e1-2}) we have $a_{k-3}=w_{2}-1$ and $a_{k-3}-1\in A$, a contradiction.
If $w=3$, then by (\ref{e1-2}) and (\ref{lem2.6-3}) we have $a_{k-3}=w_{2}-1$ and $a_{k-3}-2\in A$, a contradiction.
If $w\geqslant 5$, then by (\ref{e1-2}) and (\ref{lem2.6-3}) we have $a_{k-3}=w_{2}-1$ and $a_{k-3}-1\in A$, a contradiction.

This completes the proof of Lemma \ref{L4}.
\end{proof}

\begin{lem}\label{L3}
Let $A$ be a set of $k\geqslant 9$ integers such that $a_{k-3}=2k-7$, $a_{k-2}=2k-4$, $a_{k-1}=2k-3$ and $|2^{\wedge}A|=3k-6$.
If $2k-8\notin 2^{\wedge}A$ and there exists a positive integer $\alpha$ such that
$$a_{s}+\alpha\in 2^{\wedge}(A+\alpha), \ \ s=2, \ldots, k-1, $$
then $a_{k-4}+2k\notin 2^{\wedge}A$.
\end{lem}

\begin{proof}
Since $2k-8\notin 2^{\wedge}A$, we have
$$|\{i, 2k-8-i\}\cap A|\leqslant 1, \ \ i=0, \ldots, k-4, $$
thus,
$$k-3=|A\backslash \{a_{k-3}, a_{k-2}, a_{k-1}\}|=|[0, 2k-8]\cap A|\leqslant k-3, $$
it follows that
\begin{equation}\label{e1-9}|\{i, 2k-8-i\}\cap A|=1, \ \ i=0, \ldots, k-4. \end{equation}

Assume that $a_{k-4}+2k\in 2^{\wedge}A$. Then
\begin{equation}\label{e3-3}a_{k-4}=2k-11 \text{ or } 2k-10, \end{equation}
thus,
\begin{equation}\label{e3-4}2k-9, \ \ 2k-8\not\in A. \end{equation}
By (\ref{e1-9}) and (\ref{e3-4})  we have
\begin{equation}\label{e1-8}a_{1}=1. \end{equation}

By Lemma \ref{LL3-7} we have $|W|=1$ or $2$. Let $w\in W$. Now, we shall show that
$$6\leqslant w\leqslant 2k-11. $$

If $w=2k-5$, then by $a_{k-3}=2k-7$, (\ref{e3-4}), (\ref{e1-8}) and (\ref{e1-2}) we have $a_{2}=3$, $a_{3}=4$.
Since there exists a positive integer $\alpha$ such that $3+\alpha\in 2^{\wedge}(A+\alpha)$,
thus, $\alpha=2$, it follows that $4+\alpha\not\in 2^{\wedge}(A+\alpha)$, a contradiction.

Since $1, 2k-7\in A$ we have $w\neq 2k-7, 2k-6$.

Noting that $2k-8+a_{k-1}=a_{k-3}+a_{k-2}\in 2^{\wedge}A$ we have $w\neq 2k-8$.

If $w=2k-9$, then by (\ref{e1-3}) we have  $2k-6=\frac{w+a_{k-1}}{2}\in A$, which is impossible. Thus, $w\neq 2k-9$.

If $w=2k-10$, then by (\ref{e1-2}) and (\ref{e1-8}) we have $2k-11\not\in A$, which contradicts with (\ref{e3-3}).

By (\ref{e1-8}) we have $w\neq1$. If $w=2$ or $3$, then by $a_{k-3}=2k-7$, $a_{k-2}=2k-4$ and (\ref{e1-2}), (\ref{e1-3}) we have $a_{2}=w+2$, $a_{3}=w+3$.
Since there exists a positive integer $\alpha$ such that $w+2+\alpha\in 2^{\wedge}(A+\alpha)$, we have $\alpha=w+1$, thus, $w+3+\alpha\not\in 2^{\wedge}(A+\alpha)$,
a contradiction.

If $w=4$, then by $a_{1}=1$ and (\ref{e1-2}), we have $a_{2}=2$ and $3\notin A$.
By $a_{k-3}=2k-7$, $a_{k-2}=2k-4$, (\ref{e3-4}) and (\ref{e1-3}), we have $a_{3}=6$, $a_{4}=7$, $a_{5}=9$,
thus, $3, 8\notin A$ and
$$|S(3)\cap 2^{\wedge}A|=|S(8)\cap 2^{\wedge}A|=2, $$
which contradicts with Lemma \ref{LL3-7}.

If $w=5$, then by $a_{1}=1$ and (\ref{e1-2}), we have $a_{2}=2$ or $3$.
By $a_{k-3}=2k-7$, $a_{k-2}=2k-4$, (\ref{e3-4}) and (\ref{e1-3}), we have $a_{3}=7$, $a_{4}=8$, $a_{5}=10$.
If $a_{2}=2$, then $3, 9\notin A$ and
$$|S(3)\cap 2^{\wedge}A|=|S(9)\cap 2^{\wedge}A|=2, $$
which contradicts with Lemma \ref{LL3-7}. Thus, $a_{2}=3$.
Since there exists a positive integer $\alpha$ such that $3+\alpha\in 2^{\wedge}(A+\alpha)$, we have $\alpha=2$, thus, $7+\alpha\not\in 2^{\wedge}(A+\alpha)$, a contradiction.

Hence,
$$6\leqslant w\leqslant 2k-11. $$
By (\ref{e1-8}) we have $w-1\notin A$.
By $a_{k-3}=2k-7$, $a_{k-2}=2k-4$ and (\ref{e1-3}) we have
\begin{equation}\label{e3-5}w+1\not\in A, \ \ w+2, \ \ w+3\in A. \end{equation}
Thus,
\begin{equation}\label{e3-8.1}w-2+a_{k-1}=(w+2)+a_{k-3}\in 2^{\wedge}A, \end{equation}\vskip-5mm
\begin{equation}\label{e3-8}w-1+a_{k-1}=(w+3)+a_{k-3}\in 2^{\wedge}A, \end{equation}
$$w+1+a_{k-1}=(w+2)+a_{k-2}\in 2^{\wedge}A. $$

Next, we will show that $w+1\not\in 2^{\wedge}A$. Otherwise, $|S(w+1)\cap 2^{\wedge}A|=2$.
By Lemma \ref{LL3-7} (2) and (\ref{e3-8}) we have
$$w-1\notin 2^{\wedge}A. $$
By Lemma \ref{LL3-8} (3), we have
$$[0,w]\cap A=\left[0, \frac{w-1}{2}\right]. $$
Since there exists a positive integer $\alpha$ such that $2+\alpha\in 2^{\wedge}(A+\alpha)$, we have $\alpha=1$.
Combining with (\ref{e3-5}) we know that
$$\left[1, \frac{w+1}{2}\right]\cup \{w+3\}\subseteq A+1, $$
thus, $w+3\not\in 2^{\wedge}(A+1)$, a contradiction. Hence,
$$w+1\notin 2^{\wedge}A. $$

If $2\mid w$, then by Lemma \ref{LL3-8} (1) we have
$$[0,w]\cap A=\left[0, \frac{w}{2}\right]. $$
Since there exists a positive integer $\alpha$ such that $2+\alpha\in 2^{\wedge}(A+\alpha)$, we have $\alpha=1$.
Combining with (\ref{e3-5}) we know that
$$\left[1, \frac{w}{2}+1\right]\cup \{w+3\}\subseteq A+1, $$
thus, $w+3\not\in 2^{\wedge}(A+1)$, a contradiction.

Hence, $2\nmid w$. If $2\in A$, then by (\ref{e1-9}) we have $2k-10\not\in A$, combining with (\ref{e3-3}) we have $a_{k-4}=2k-11$, again by (\ref{e1-9}) we have $3\not\in A$.
By Lemma \ref{LL3-8} (2) we have
$$[0,w]\cap A=\{0, 1, 2\}\cup \left[\frac{w+1}{2}, w-3\right], $$
it implies that $w-2, w-1\notin A$ and
$$w-2, \ \ w-1\in 2^{\wedge}A. $$
Combining with (\ref{e3-8.1}) and (\ref{e3-8}) we have
$$|S(w-2)\cap 2^{\wedge}A|=|S(w-1)\cap 2^{\wedge}A|=2, $$
which contradicts with Lemma \ref{LL3-7}.

Hence, $2\not\in A$. By Lemma \ref{LL3-8} (2) we have
$$[0,w]\cap A=\{0, 1\}\cup \left[\frac{w+1}{2}, w-2\right]. $$
Since there exists a positive integer $\alpha$ such that $\frac{w+1}{2}+\alpha\in 2^{\wedge}(A+\alpha), $
we have $\alpha=\frac{w-1}{2}$, thus,
$$\left\{\frac{w-1}{2}, \frac{w+1}{2}\right\}\cup \left[w, w-2+\frac{w-1}{2}\right]\subseteq A+\frac{w-1}{2}, $$
it follows that $w+2\not\in 2^{\wedge}\left(A+\frac{w-1}{2}\right)$, a contradiction.

This completes the proof of Lemma \ref{L3}.
\end{proof}

Similar to the proof of Lemma \ref{L3}, we have the following result.

\begin{lem}\label{L5}
Let $A$ be a set of $k\geqslant 9$ integers such that $a_{k-3}=2k-7$, $a_{k-2}=2k-4$, $a_{k-1}=2k-3$ and $|2^{\wedge}A|=3k-6$.
If $2k-8\notin 2^{\wedge}A$ and there exists a positive integer $\alpha$ such that
$$a_{2}=2a_{1}, \ \ a_{2}\neq a_{1}+\alpha, \ \ a_{s}+\alpha\in 2^{\wedge}(A+\alpha), \ \ s=3, \ldots, k-1, $$
then $a_{k-4}+2k\notin 2^{\wedge}A$.
\end{lem}

\section{Results on locally dense sets}\label{S4}
In this section, let $A$ be a finite set of $k\geqslant 3$ integers such that
\begin{equation}\label{ee}a_{i}<2i, \ \ i=1, \ldots, k-2, \ \ a_{k-1}\geqslant 2k-2, \ \ \gcd(A)=1. \end{equation}
Write $b_{0}=0$ and
$$B=[1, 2k-4]\backslash 2^{\wedge}(A\backslash\{a_{k-1}\}):=\{b_{1}, b_{2}, \ldots, b_{m}\}, $$
where $b_{1}<b_{2}<\cdots<b_{m}$.

Noting that
$$[1, 2k-4]\backslash B, \;a_{k-1}+A\backslash\{a_{k-1}\}\subseteq 2^{\wedge}A, $$
we have
\begin{equation}\label{ee2}|2^{\wedge}A|\geqslant |[1, 2k-4]\backslash B|+|a_{k-1}+A\backslash\{a_{k-1}\}|=3k-5-m. \end{equation}

By the pigeonhole principle, we have the following facts:
\begin{itemize}
  \item If $2k-3\not\in 2^{\wedge}A$, then \begin{equation}\label{eq2}|\{i, 2k-3-i\}\cap A|=1, \ \ i=1, \ldots, k-2.\end{equation}
  \item If $a_{2}=3$ and $2k-2\not\in 2^{\wedge}A$, then \begin{equation}\label{eq3}|\{i, 2k-2-i\}\cap A|=1, \ \ i=3, \ldots, k-1.\end{equation}
  \item If $a_{2}=3$ and $2k-1\not\in 2^{\wedge}A$, then \begin{equation}\label{eq4}|\{i, 2k-1-i\}\cap A|=1, \ \ i=4, \ldots, k-1.\end{equation}
  \item If $a_{3}=5$ and $2k\not\in 2^{\wedge}A$, then \begin{equation}\label{eq5}|\{i, 2k-i\}\cap A|=1, \ \ i=5, \ldots, k.\end{equation}
  \item If $a_{3}=5$ and $2k+1\notin 2^{\wedge}A$, then \begin{equation}\label{eq6}|\{i, 2k+1-i\}\cap A|=1, \ \ i=6, \ldots, k.\end{equation}
\end{itemize}

\begin{lem}\label{P4}
Let $A$ be as in (\ref{ee}). If $a_{s}+a_{k-1}\in 2^{\wedge}(A\backslash\{a_{k-1}\})$ for all $0\leqslant s\leqslant k-4$, then
$|2^{\wedge}A|\geqslant 3k-6. $
\end{lem}

\begin{proof}
By Theorem F we know that $|2^{\wedge}A|\geqslant3k-7$.
Assume that $|2^{\wedge}A|=3k-7$. If $k\equiv 0\pmod{3}$, then by Theorem F we have $a_{\frac{2k-6}{3}}=k-3\equiv0\pmod{3}$ and $a_{i}\equiv 1\pmod{3}$ for all $\frac{2k-3}{3}\leqslant i\leqslant k-1$,
it follows that
$$a_{\frac{2k-6}{3}}+a_{k-1}\not\in 2^{\wedge}(A\backslash\{a_{k-1}\}), $$
a contradiction.
If $k\equiv 1\pmod{3}$, then by Theorem F we have $a_{\frac{2k-8}{3}}=k-3\equiv1\pmod{3}$ and $a_{i}\equiv 0\pmod{3}$ for all $\frac{2k-5}{3}\leqslant i\leqslant k-1$,
it follows that
$$a_{\frac{2k-8}{3}}+a_{k-1}\not\in 2^{\wedge}(A\backslash\{a_{k-1}\}), $$
a contradiction.

Therefore, $|2^{\wedge}A|\geqslant 3|A|-6$.

This completes the proof of Lemma \ref{P4}.
\end{proof}

\begin{lem}[\cite{Wang2025}, Lemma 2.2]\label{L1}
For any integer $b\in B$, we have

(1) $b\not\in A\backslash\{a_{k-1}\}$, $2\mid b$, $\frac{b}{2}\in A\backslash\{a_{k-1}\}$;

(2) $\left|[0, b]\cap (A\backslash\{a_{k-1}\})\right|=\frac{b}{2}+1$. Moreover, $\left|\left\{i, b-i\right\}\cap (A\backslash\{a_{k-1}\})\right|=1$ for all $i=0, 1, \ldots, \frac{b}{2}$;

(3) Let $u$, $b$ be integers such that $b<k-2$ and $1\leqslant u\leqslant b$. If $2k-4+u\not\in 2^{\wedge}(A\backslash\{a_{k-1}\})$, then
$$|[b+1, 2k-5+u-b]\cap (A\backslash\{a_{k-1}\})|=k-2+\left\lfloor\frac{u}{2}\right\rfloor-b, $$
$$|\{i, 2k-4+u-i\}\cap (A\backslash\{a_{k-1}\})|=1, \ \ i=b+1, \ldots, k-2+\left\lfloor \frac{u}{2}\right\rfloor. $$

(4) If $b<2k-4$, then $a_{\frac{b}{2}+1}=b+1\in (A\backslash\{a_{k-1}\})$.
\end{lem}

\begin{lem}[\cite{Wang2025}, Lemma 2.3]\label{L2}
Assume that $m\geqslant 2$. For any $i\in \{0, \ldots, m-1\}$, we have $b_{i+1}\geqslant 2b_{i}+2$.
\end{lem}

\begin{lem}\label{P18} Assume that $m\geqslant 3$. Let $k\geq 5$ be an integer and $A$ be as in (\ref{ee}) with $a_{3}=5$. Then
$$|2^{\wedge}A|\geqslant 3k-5-m+\left\lfloor\frac{b_{m-1}}{4}\right\rfloor. $$
In particular, if $b_{2}\geqslant 12$, $m\geqslant 3$ or $b_{2}\geqslant 10$, $m\geqslant 4$, then $|2^{\wedge}A|\geqslant 3k-5$.
\end{lem}

\begin{proof} Noting that
$$[1, 2k-4]\cap 2^{\wedge}(A\backslash\{a_{k-1}\}), \ \ 2k-4+[1, b_{m-1}], $$
$$\{2k-3+b_{m-1}, 2k-2+b_{m-1}\}, \ \ a_{k-1}+[b_{m-1}+1, a_{k-2}]\cap A $$
are  mutually disjoint parts. We have
$$|[1, 2k-4]\cap 2^{\wedge}(A\backslash\{a_{k-1}\})|=|[1, 2k-4]\backslash B|=2k-4-m. $$
By $a_{3}=5$ and Lemma 2.7 of \cite{Wang2025} we have $2k-4+[1, b_{m-1}]$ provides at least $\frac{b_{m-1}}{2}+\left\lfloor\frac{b_{m-1}}{4}\right\rfloor$ integers of $2^{\wedge}A$.
By Proposition 2.8 of \cite{Wang2025}, at least one of $2k-3+b_{m-1}$ and $2k-2+b_{m-1}$ belongs to $2^{\wedge}(A\backslash\{a_{k-1}\})$.
By Lemma \ref{L1} (2), we have
$$\left|[0, b_{m-1}]\cap (A\backslash\{a_{k-1}\})\right|=\frac{b_{m-1}}{2}+1, $$
thus,
$$|[b_{m-1}+1, a_{k-2}]\cap (A\backslash\{a_{k-1}\})|=k-2-\frac{b_{m-1}}{2}. $$
Hence, $a_{k-1}+[b_{m-1}+1, a_{k-2}]\cap A$ provides $k-2-\frac{b_{m-1}}{2}$ integers of $2^{\wedge}A$.
To sum up, we obtain at least $3k-5-m+\left\lfloor\frac{b_{m-1}}{4}\right\rfloor$ integers of $2^{\wedge}A$.

By Lemma \ref{L2} we have
$$b_{m-1}\geqslant 2b_{m-2}+2\geqslant \cdots\geqslant 2^{m-3}b_{2}+\sum_{i=1}^{m-3}2^{i}=2^{m-3}b_{2}+2^{m-2}-2. $$
If $b_{2}\geqslant 12$, $m\geqslant 3$ or $b_{2}\geqslant 10$, $m\geqslant 4$, then $\left\lfloor\frac{b_{m-1}}{4}\right\rfloor\geqslant m$,
thus, $|2^{\wedge}A|\geqslant 3k-5$.

This completes the proof of Lemma \ref{P18}.
\end{proof}

\begin{pro}\label{P16}
Let $k\geqslant 7$ be an integer and $A$ be as in (\ref{ee}) with $a_{2}=3$, $a_{3}=5$.
If $(a_{4}, a_{5})\neq (6, 8)$, then $|2^{\wedge}A|\geqslant 3k-5$.
\end{pro}

\begin{proof}
Since $a_{4}<8$, $a_{5}<10$, it is sufficient to consider the following cases:

(i) $a_{1}=1$, $a_{2}=3$, $a_{3}=5$, $a_{4}=6$, $a_{5}=7$;

(ii) $a_{1}=1$, $a_{2}=3$, $a_{3}=5$, $a_{4}=6$, $a_{5}=9$;

(iii) $a_{1}=1$, $a_{2}=3$, $a_{3}=5$, $a_{4}=7$, $a_{5}=8$;

(iv) $a_{1}=1$, $a_{2}=3$, $a_{3}=5$, $a_{4}=7$, $a_{5}=9$.

It is easy to see that if $k=7, 8$, then $|2^{\wedge}A|\geqslant 3k-5$. Assume that $k\geqslant 9$.

By Lemma \ref{L1} we have $2\mid b$ for all $b\in B$. If $m\geqslant 3$, then $b_2\geqslant 14$. By Lemma \ref{P18} we have $|2^{\wedge}A|\geqslant 3k-5$.

Assume that $m=1$ or $2$. The case (i)-(iii) are similar, we only prove (i). We divide into the following cases:

{\bf Case 1}. $2k-3\not\in 2^{\wedge}A$. Since $2, 4\not\in A$, thus by (\ref{eq2}) we have
$$a_{k-2}=2k-5, \ \ a_{k-3}=2k-7.$$
Since  $a_{1}=1$, $a_{2}=3$, $a_{3}=5$, $a_{4}=6$, $a_{5}=7$, we have
$$2k-2, \ \ 2k-1, \ \ 2k, \ \ 2k+1, \ \ 2k+2\in 2^{\wedge}A. $$
If $a_{k-1}\geqslant 2k$, then $2k-2, 2k-1\not\in a_{k-1}+A\backslash\{a_{k-1}\}$.
If $a_{k-1}=2k-1$, then $2k-2, 2k+1\not\in a_{k-1}+A\backslash\{a_{k-1}\}$.
If $a_{k-1}=2k-2$, then $2k, 2k+2\not\in a_{k-1}+A\backslash\{a_{k-1}\}$.
Combining with (\ref{ee2}) we have $|2^{\wedge}A|\geqslant 3k-5$.

{\bf Case 2}. $2k-3\in 2^{\wedge}A$. Clearly,
$$2k-3\not\in a_{k-1}+A\backslash\{a_{k-1}\}. $$
If $m=1$, then combining with (\ref{ee2}) we have $|2^{\wedge}A|\geqslant 3k-5$.

Assume that $m=2$. We divide into the following two cases.

{\bf Case 2.1}. $a_{k-1}\geqslant 2k-1$. If $2k-2\in 2^{\wedge}A$, then by (\ref{ee2}) and $2k-3, 2k-2\not\in a_{k-1}+A\backslash\{a_{k-1}\}$,
we have $|2^{\wedge}A|\geqslant 3k-5$.

Assume that $2k-2\not\in 2^{\wedge}A$. Noting that $4\not\in A$, thus by (\ref{eq3}) we have $a_{k-2}=2k-6$. Since $a_{3}=5$, $a_{4}=6$, $a_{5}=7$, we have
$$2k-1, \ \ 2k, \ \ 2k+1\in 2^{\wedge}A. $$
If $a_{k-1}\geqslant 2k$, then $2k-1\not\in a_{k-1}+A\backslash\{a_{k-1}\}$.
If $a_{k-1}=2k-1$, then $2k+1\not\in a_{k-1}+A\backslash\{a_{k-1}\}$.
Combining with (\ref{ee2}) we have $|2^{\wedge}A|\geqslant 3k-5$.

{\bf Case 2.2}. $a_{k-1}=2k-2$. We shall show that $\{2k, 2k+2\}\cap 2^{\wedge}A\neq \emptyset$. Otherwise, we have
\begin{equation}\label{P16-2}|\{i, 2k+2-i\}\cap A|\leqslant 1, \ \ i=6, \ldots, k+1. \end{equation}

By (\ref{eq5}) we have $k\in A$. If $k$ is even, then
$$k\in A\overset{(\ref{P16-2})}\longrightarrow k+2\not\in A\overset{(\ref{eq5})}\longrightarrow k-2\in A\longrightarrow\cdots\longrightarrow 6\in A, $$
thus,
$$\{0, 1, 3, 5, 7\}\cup 2\left[3, \frac{k}{2}\right]\subseteq A\backslash\{a_{k-1}\}. $$
If $k+1\in A\backslash\{a_{k-1}\}$, then $[1, 2k-4]\backslash 2^{\wedge}(A\backslash\{a_{k-1}\})=\{2\}$, a contradiction.
Hence, $k+1\notin A$. By (\ref{eq5}) we have $k-1\in A$, thus, $[1, 2k-4]\backslash 2^{\wedge}(A\backslash\{a_{k-1}\})=\{2\}$,
which is also impossible.

If $k$ is odd, then
$$k\in A\overset{(\ref{P16-2})}\longrightarrow k+2\not\in A\overset{(\ref{eq5})}\longrightarrow k-2\in A\longrightarrow\cdots\longrightarrow 7\in A\longrightarrow 5\in A. $$
Combining with $a_{0}=0$, $a_{1}=1$, $a_{2}=3$, we have
$$2\left[0, \frac{k-1}{2}\right]+1\subseteq A\backslash\{a_{k-1}\}. $$
If $k+1\in A\backslash\{a_{k-1}\}$, then $[1, 2k-4]\backslash 2^{\wedge}(A\backslash\{a_{k-1}\})=\{2\}$, a contradiction.
Hence, $k+1\notin A$. By (\ref{eq5}) we have $k-1\in A$, thus, $[1, 2k-4]\backslash 2^{\wedge}(A\backslash\{a_{k-1}\})=\{2\}$,
which is also impossible.

So,
$$\{2k, 2k+2\}\cap 2^{\wedge}A\neq \emptyset. $$
Since $2, 4\notin A$, we have $2k, 2k+2\not\in a_{k-1}+A\backslash\{a_{k-1}\}$.
By (\ref{ee2}) we have $|2^{\wedge}A|\geqslant 3k-5$.

Now, we prove (iv). If $a_{i}=2i-1$ for all $i=1, \ldots, k-2$, then
$$A=\{0\}\cup (2\left[1, k-2\right]-1)\cup\{a_{k-1}\}, $$
thus, $|2^{\wedge}A|\geqslant 3k-5$.

Let $j$ be the smallest integer such that $a_{j}\neq 2j-1$. Then $j\geqslant 6$, thus, $2j-6\notin A$.
Since $a_{j-1}=2j-3$ and $a_{j}<2j$, we have $a_{j}=2j-2$.

The proof of case $2k-3\in 2^{\wedge}A$ is similar to Case 2, we omit it.

Now we assume that $2k-3\not\in 2^{\wedge}A$. Since $2, 4\not\in A$, thus by (\ref{eq2}) we have $a_{k-2}=2k-5$, $a_{k-3}=2k-7$.
Since  $a_{1}=1$, $a_{2}=3$, $a_{3}=5$, $a_{4}=7$, $a_{5}=9$ and $a_{j}=2j-2$, we have
$$2k-2, \ \ 2k, \ \ 2k+2, \ \ 2k+4, \ \ 2k+2j-7\in 2^{\wedge}A. $$
If $a_{k-1}\geqslant 2k+1$, then $2k-2, 2k\not\in a_{k-1}+A\backslash\{a_{k-1}\}$.
If $a_{k-1}=2k$, then $2k-2, 2k+2\not\in a_{k-1}+A\backslash\{a_{k-1}\}$.
If $a_{k-1}=2k-1$, then by $2j-6\notin A$ we have $2k-2, 2k+2j-7\not\in a_{k-1}+A\backslash\{a_{k-1}\}$.
If $a_{k-1}=2k-2$, then $2k, 2k+2\not\in a_{k-1}+A\backslash\{a_{k-1}\}$.
By (\ref{ee2}) we have $|2^{\wedge}A|\geqslant 3k-5$.

This completes the proof of Proposition \ref{P16}.
\end{proof}

\begin{lem}\label{L16}
Let $k\geqslant 10$ be an integer and $A$ be as in (\ref{ee}) satisfying $a_{2}=3$, $a_{3}=5$, $a_{4}=6$, $a_{5}=8$ and $a_{6}=11$.
If there exists an integer $n\geqslant \max(2k, a_{k-1}+1)$ such that $n+a_{i}\in 2^{\wedge}A$ for all $i=0, \ldots, k-4$, then $2k-3, 2k-2, 2k-1\in 2^{\wedge}A$.
\end{lem}

\begin{proof}
Assume that $2k-3\notin 2^{\wedge}A$. If $k=10$, then $2k-3=a_{4}+a_{6}\in 2^{\wedge}A$, a contradiction.
If $k=11$, then $2k-3=a_{5}+a_{6}\in 2^{\wedge}A$, a contradiction.

Let $k\geqslant 12$. Then by (\ref{eq2}) we have
$$a_{k-2}=2k-5, \ \ a_{k-3}=2k-7, \ \ a_{k-4}=2k-10, \ \ a_{k-5}=2k-12, \ \ a_{k-6}=2k-13. $$
Since $n+a_{i}\in 2^{\wedge}A$ for $i=0, \ldots, k-4$, we have
$$n+a_{k-4}=a_{k-3}+a_{k-1}, \ \ a_{k-3}+a_{k-2}  \text{ or }  a_{k-2}+a_{k-1}. $$
By $n\geqslant 2k$ we have
$$n=a_{k-1}+3 \text{ or } a_{k-1}+5. $$
Since $a_{k-1}\geqslant 2k-2$, we have $a_{k-5}+n\geqslant 4k-11$,
thus, $n=a_{k-1}+5\geqslant 2k+3$ and
$$a_{k-5}+n=a_{k-1}+a_{k-3}. $$
Hence, $a_{k-6}+n\notin 2^{\wedge}A$, a contradiction. Therefore, $2k-3\in 2^{\wedge}A$.

Similarly, by (\ref{eq3}), (\ref{eq4}) we can show that $2k-2, 2k-1\in 2^{\wedge}A$.

This completes the proof of Lemma \ref{L16}.
\end{proof}

\begin{pro}\label{P16.1}
Let $k\geqslant 7$ be an integer and $A$ be as in (\ref{ee}) satisfying $a_{2}=3$, $a_{3}=5$, $a_{4}=6$, $a_{5}=8$.
If there exists an integer $n\geqslant \max(2k, a_{k-1}+1)$ such that $n+a_{i}\in 2^{\wedge}A$ for all $i=0, \ldots, k-4$,
then $|2^{\wedge}A|\geqslant 3k-5$.
\end{pro}

\begin{proof}
Since $n+a_{i}\in 2^{\wedge}A$ for $i=0, \ldots, k-4$, we have
\begin{equation}\label{eqb1}n+a_{k-4}=a_{k-3}+a_{k-1}, \ \ a_{k-3}+a_{k-2}  \text{ or }  a_{k-2}+a_{k-1}. \end{equation}

If $k=7$, then $A=\{0, 1, 3, 5, 6, 8, a_{6}\}$. By $n\geqslant 14$ and (\ref{eqb1}), we have $n=a_{6}+1$ or $a_{6}+3$,
it implies that either $n+a_{2}\not\in 2^{\wedge}A$ or $n+a_{1}\not\in 2^{\wedge}A$, a contradiction.

By a similar discussion, it is also impossible for $k=8, 9$.

Now we assume that $k\geqslant 10$. We divide into the following two cases:

{\bf Case 1. }$m=1$ or $2$. If $2k-3\not\in 2^{\wedge}A$, then by $2, 4\not\in A$ and (\ref{eq2}) we have $a_{k-2}=2k-5$, $a_{k-3}=2k-7$.
Since $a_{1}=1$, $a_{2}=3$, $a_{3}=5$, $a_{4}=6$, $a_{5}=8$, we have
$$2k-2, 2k-1, 2k, 2k+1, 2k+3\in 2^{\wedge}A. $$
If $a_{k-1}\geqslant 2k$, then $2k-2, 2k-1\not\in a_{k-1}+A\backslash\{a_{k-1}\}$.
If $a_{k-1}=2k-1$, then $2k-2, 2k+1\not\in a_{k-1}+A\backslash\{a_{k-1}\}$. By (\ref{ee2}) we have $|2^{\wedge}A|\geqslant 3k-5$.

Assume that $a_{k-1}=2k-2$. Then
$$2k, 2k+2, 2k+5\not\in a_{k-1}+A\backslash\{a_{k-1}\}. $$
Noting that $a_{6}<12$. If $a_{6}=11$, then by Lemma \ref{L16} we have $2k-3\in 2^{\wedge}A$, a contradiction.
Hence, $a_{6}=9$ or $10$, so, $2k+2\in 2^{\wedge}A$ or $2k+5\in 2^{\wedge}A$. By (\ref{ee2}) we have $|2^{\wedge}A|\geqslant 3k-5$.

Assume that $2k-3\in 2^{\wedge}A$. Clearly,
$$2k-3\not\in a_{k-1}+A\backslash\{a_{k-1}\}. $$
If $m=1$, then by (\ref{ee2}) we have $|2^{\wedge}A|\geqslant 3k-5$.

Assume that $m=2$. For case $a_{k-1}=2k-2$, the proof is similar to Proposition \ref{P16}, we omit it.
Assume that $a_{k-1}\geqslant 2k-1$. If $2k-2\in 2^{\wedge}A$, then by (\ref{ee2}) and
$$2k-3, 2k-2\not\in a_{k-1}+A\backslash\{a_{k-1}\}, $$
we have $|2^{\wedge}A|\geqslant 3k-5$.

Assume that $2k-2\not\in 2^{\wedge}A$. By $4, 7\not\in A$ and (\ref{eq3}) we have
$$a_{k-2}=2k-6, \ \ a_{k-3}=2k-9. $$
Since $a_{3}=5$, $a_{4}=6$, $a_{5}=8$, we have
$$2k-1, 2k, 2k+2\in 2^{\wedge}A. $$
If $a_{k-1}\geqslant 2k$, then $2k-1\not\in a_{k-1}+A\backslash\{a_{k-1}\}$, thus, $|2^{\wedge}A|\geqslant 3k-5$.
Assume that $a_{k-1}=2k-1$. Then
$$2k+1, 2k+3\not\in a_{k-1}+A\backslash\{a_{k-1}\}. $$
Noting that $a_{6}<12$. If $a_{6}=11$, then by Lemma \ref{L16} we have $2k-2\in 2^{\wedge}A$, a contradiction.
Hence, $a_{6}=9$ or $10$, so, $2k+1\in 2^{\wedge}A$ or $2k+3\in 2^{\wedge}A$. By (\ref{ee2}) we have $|2^{\wedge}A|\geqslant 3k-5$.

{\bf Case 2. }$m\geqslant 3$. By $b_{2}\geqslant 10$ and Lemma \ref{P18}, it is sufficient to consider
$$b_{2}=10, \ \ m=3. $$
Thus, $9, 10\not\in A$. By $a_{6}<12$ we have $a_{6}=11$. By Lemma \ref{L16}, we have
$$2k-3, 2k-2, 2k-1\in 2^{\wedge}A. $$

If $a_{k-1}\geqslant 2k$, then $2k-3, 2k-2, 2k-1\notin a_{k-1}+A\backslash\{a_{k-1}\}$, thus, $|2^{\wedge}A|\geqslant 3k-5$.

If $a_{k-1}=2k-1$, then
$$2k-3, 2k-2, 2k+1, 2k+3, 2k+6\notin a_{k-1}+A\backslash\{a_{k-1}\}. $$
If $2k+1\in 2^{\wedge}A$, then $|2^{\wedge}A|\geqslant 3k-5$. Assume that $2k+1\notin 2^{\wedge}A$.
By (\ref{eq6}) we have
$$a_{k-2}=2k-6, \ \ a_{k-3}=2k-8, $$
thus, $2k+3=a_{k-3}+a_{6}\in 2^{\wedge}A$, it follows that $|2^{\wedge}A|\geqslant 3k-5$.

Assume that $a_{k-1}=2k-2$. Then
$$2k-3, 2k, 2k+2, 2k+5, 2k+7, 2k+8\notin a_{k-1}+A\backslash\{a_{k-1}\}. $$
If $2k\in 2^{\wedge}A$, put $x=2k$. Assume that $2k\notin 2^{\wedge}A$.
By (\ref{eq5}) we have $a_{k-2}=2k-7$, $a_{k-3}=2k-9$, thus, $2k+2=a_{k-3}+a_{6}\in 2^{\wedge}A$. In this case, put $x=2k+2$.

Put $d=n-a_{k-1}=n-(2k-2)\geqslant2$. We shall show that there exists an integer $p\in[0,k-4]$ such that
\begin{equation}\label{eq:P16.1-new}
a_p+d\notin A
\end{equation}
and
$$n+a_p>2k+2. $$

If $d=2$, then $a_{3}+d=5+2=7\notin A$. If $d=3$, then $a_{4}+d=6+3=9\notin A$.
If $d=4$, then $a_{2}+d=3+4=7\notin A$.
If $d=5$, then $a_{3}+d=5+5=10\notin A$.
Thus, (\ref{eq:P16.1-new}) holds for $2\leqslant d\leqslant5$, and in these four cases we have respectively
$$n+a_p=2k+5,\quad 2k+7,\quad 2k+5,\quad 2k+8. $$
Hence,
$$n+a_p>2k+2. $$

Assume that $d\geqslant6$. If
$$a_i+d\in A,\qquad i=0,\ldots,k-4, $$
then
$$\{a_i+d:0\leqslant i\leqslant k-4\}\subseteq A\cap[6,a_{k-1}]. $$
Since $A\cap[0,5]=\{0,1,3,5\}$, we have
$$|A\cap[6,a_{k-1}]|=k-4. $$
On the other hand,
$$|\{a_i+d:0\leqslant i\leqslant k-4\}|=k-3, $$
a contradiction. Thus, there exists an integer $p\in[0,k-4]$ such that
$$a_p+d\notin A. $$
Moreover,
$$n+a_p\geqslant n=a_{k-1}+d\geqslant2k+4>2k+2. $$

By the assumption of Proposition \ref{P16.1}, we have $n+a_p\in2^{\wedge}A$.
If $n+a_p=a_{k-1}+a_j$ for some $0\leqslant j\leqslant k-2$, then
$$a_j=n+a_p-a_{k-1}=a_p+d, $$
which contradicts with $a_p+d\notin A$.
Hence,
$$n+a_p\notin a_{k-1}+A\backslash\{a_{k-1}\}. $$

Since $n+a_p>2k+2$, the three integers $2k-3$, $x$, $n+a_p$ are distinct. Moreover,
$$2k-3,\ x,\ n+a_p>2k-4 $$
and
$$2k-3,\ x,\ n+a_p
\notin a_{k-1}+A\backslash\{a_{k-1}\}. $$
Thus, by $m=3$ and (\ref{ee2}), we have
$$|2^{\wedge}A|\geqslant 3k-5-m+3=3k-5. $$

This completes the proof of Proposition \ref{P16.1}.
\end{proof}

Similarly, we have the following proposition.

\begin{pro}\label{P16.3}
Let $k\geqslant 7$ be an integer and $A$ be as in (\ref{ee}) satisfying $a_{2}=2$, $a_{3}=5$, $a_{4}=6$, $a_{5}=9$. If there exists an integer $n\geqslant \max(2k, a_{k-1}+1)$ such that
$$n+a_{i}\in 2^{\wedge}A, \ \ i=0, \ldots, k-4. $$
then $|2^{\wedge}A|\geqslant 3k-5$.
\end{pro}

\section{Structure and inverse results on dense sets}\label{S5}
\begin{lem}[\cite{Wang2025}, Theorem 1.3]\label{LL3-4}
Let $A$ be a finite set of $k\geqslant 3$ integers such that $a_{k-1}=2k-3$.
Then $|2^{\wedge}A|=3k-7$ if and only if $k\geqslant 4$ and one of the following cases holds:

(1) $A=[0, \theta-k+1]\cup [\theta, 2k-3]$, $\theta=k, k+1, \ldots, 2k-4$;

(2) $A=\{2i, 2j-1: i\in [0, \theta], j\in[\theta+1, k-1]\}$, $\theta=1, 2, \ldots,  k-3$;

(3) $A=\{3i, 3j-k: i\in [0, \theta], j\in[\theta+1, k-1]\}$, $\theta=\left\lfloor\frac{k-3}{3}\right\rfloor+1, \ldots, \left\lfloor\frac{2k-3}{3}\right\rfloor$, $3\nmid k$;

(4) $A=4\left[0, \frac{k-2}{2}\right]\cup \left(4\left[1, \frac{k}{2}\right]-3\right)$, $2\mid k$;

(5) $A=\left\{3i: 0\leqslant i\leqslant \frac{2k-3}{3}\right\}\cup\left\{\theta+3i: 0\leqslant i\leqslant \frac{k-3}{3}\right\}$, $\theta=1, \ldots, k-1$, $3\nmid\theta$, $3\mid k$;

(6) $A=\{0, 1, 4, 5, 6, 9\}$, $\{0, 3, 4, 5, 8, 9\}$, $\{0, 1, 2, 5, 6, 7, 11\}$, $\{0, 1, 3, 4, 7, 8, 11\}$,

$\{0, 1, 4, 5, 6, 10, 11\}$, $\{0, 1, 4, 5, 7, 8, 11\}$, $\{0, 1, 5, 6, 7, 10, 11\}$, $\{0, 3, 4, 6, 7, 10, 11\}$,

$\{0, 3, 4, 7, 8, 10, 11\}$, $\{0, 4, 5, 6, 9, 10, 11\}$, $\{0, 1, 2, 6, 7, 8, 12, 13\}$, $\{0, 1, 4, 5, 6, 9, 10, 13\}$,

$\{0, 1, 5, 6, 7, 8, 12, 13\}$, $\{0, 1, 5, 6, 7, 11, 12, 13\}$, $\{0, 2, 3, 5, 7, 8, 10, 13\}$,

$\{0, 2, 3, 5, 8, 10, 11, 13\}$, $\{0, 3, 4, 7, 8, 9, 12, 13\}$, $\{0, 3, 5, 6, 8, 10, 11, 13\}$,

$\{5i, 5i+\theta, 15: i=0, 1, 2\}\cup \{5-\theta, 10-\theta\}$, $\theta=1, 2, 3, 4$,

$\{5i, 5i+\theta, 15: i=0, 1, 2\}\cup \{10-\theta, 15-\theta\}$, $\theta=1, 2, 3, 4$,

$\{0, 1, 2, 6, 7, 8, 9, 14, 15\}$, $\{0, 1, 4, 5, 7, 8, 11, 12, 15\}$, $\{0, 1, 6, 7, 8, 9, 13, 14, 15\}$,

$\{0, 3, 4, 7, 8, 10, 11, 14, 15\}$, $\{0, 1, 2, 7, 8, 9, 10, 15, 16, 17\}$, $\{0, 1, 5, 6, 7, 10, 11, 12, 16, 17\}$,

$\{0, 2, 3, 5, 7, 8, 10, 12, 15, 17\}$, $\{0, 2, 5, 7, 9, 10, 12, 14, 15, 17\}$, $\{0, 3, 4, 6, 7, 10, 11, 13, 14, 17\}$.
\end{lem}

\begin{pro}\label{LL3-3}
Let $A$ be a finite set of $k\geqslant 5$ integers such that $a_{1}=1$, $a_{2}=2$ and $a_{k-1}=2k-4$.
Then $|2^{\wedge}A|=3k-7$ if and only if
$$A=\{0, 1, 2, 5, 6, 7, 11, 12\}, \{0, 1, 2, 6, 7, 8, 9, 13, 14\}, \{0, 1, 2, 3, 7, 8, 9, 10, 15, 16\}$$
or
$$A=\left[0, a\right]\cup \left[k-2+a, a_{k-1}\right], \ \ a=2, \ldots, k-3. $$
\end{pro}

\begin{proof}
(Sufficiency)It is easy to verify.

(Necessity)
Write
$$T=\{a_{i}: 1\leqslant i\leqslant k-2\}\cup\{a_{i}+a_{k-1}: 0\leqslant i\leqslant k-2\}. $$
Then $T\subseteq 2^{\wedge}A$ and $|T|=2k-3$. By the definition of $W$ we have
$$|S(b)\cap 2^{\wedge}A|\geqslant 1, \ \ b\in [0, a_{k-1}]\backslash(A\cup W), $$
thus,
$$|2^{\wedge}A|\geqslant |T|+|[0, a_{k-1}]\backslash(A\cup W)|\geqslant 3k-6-|W|. $$
By Lemma \ref{LL3-1} we have $|W|\leqslant 2$, combining with $|2^{\wedge}A|=3k-7$, we have $|W|=1$ or $2$.

{\bf Case 1. }$|W|=2$. Write $W=\{w_{1}, w_{2}\}$ with $w_{1}<w_{2}$.
By Lemma \ref{LL3-2} and $a_{k-1}=2k-4$, we have $2\mid w_{1}$, $2\mid w_{2}$. Since $a_{1}=1$, $a_{2}=2$, we have
$$4\leqslant w_{1}<w_{2}\leqslant 2k-6. $$

By (\ref{e1-2}) and $1\in A$, $w_{1}, w_{2}\geqslant 4$, we have $w_{1}-1, w_{2}-1\not\in A$.
If $\left[k-2+\frac{w_{1}}{2}, a_{k-1}\right]\subseteq A$, then
$$\left[w_{1}+1, a_{k-1}-1\right]+a_{k-1}\subseteq 2^{\wedge}A. $$
Since $w_{2}\in \left[w_{1}+1, a_{k-1}-1\right]$, we have $w_{2}+a_{k-1}\in 2^{\wedge}A$, which is impossible.
Hence, by Lemma \ref{LL3-6} (1) we have
$$|S(w_{1}-1)\cap 2^{\wedge}A|=2. $$
Combining with $|2^{\wedge}A|=3k-7$, we have
\begin{equation}\label{e1-19}|S(b)\cap 2^{\wedge}A|=1, \ \ b\in [0, a_{k-1}]\backslash(A\cup W), \ \ b\neq w_{1}-1. \end{equation}
By Lemma \ref{LL3-6} (1) and (\ref{e1-19}) we have
\begin{equation}\label{e1-23}\left[k-2+\frac{w_{2}}{2}, a_{k-1}\right]\subseteq A. \end{equation}
If $w_{2}<2k-6$, then $w_{2}+1<k-2+\frac{w_{2}}{2}$. By (\ref{e1-3}) and (\ref{e1-23}) we have $w_{2}+1\not\in A$.
By Lemma \ref{LL3-6} (2) and (\ref{e1-19}) we have
$$\left[0, \frac{w_{2}}{2}\right]\subseteq A, $$
thus, $[1, w_{2}-1]\subseteq 2^{\wedge}A$, that is, $w_{1}\in 2^{\wedge}A$, which is impossible.
Hence, $w_{2}=2k-6$.
By (\ref{e1-3}) we have $a_{k-1}-1=2k-5\in A$, again by (\ref{e1-3}) we have
\begin{equation}\label{e1-24}w_{1}+1\not\in A, \ \ w_{1}+2\in A. \end{equation}
By Lemma \ref{LL3-6} (2) and (\ref{e1-19}) we have
\begin{equation}\label{e5-1}\left[0, \frac{w_{1}}{2}\right]\subseteq A, \end{equation}
thus, by (\ref{e1-2}) and (\ref{e1-24}) we have
$$\left[\frac{w_{1}}{2}+1, w_{1}+1\right]\cap A=\emptyset. $$
By (\ref{e1-2}), (\ref{e5-1}) and $w_{2}=2k-6$ we have
\begin{equation}\label{e5-2}\left[2k-7-w_{1}, 2k-7-\frac{w_{1}}{2}\right]\subseteq A, \ \ \left[2k-6-\frac{w_{1}}{2}, w_{2}\right]\cap A=\emptyset. \end{equation}

If $w_{1}+2\neq \frac{w_{2}}{2}$, then by $w_{1}+2\in A$ and (\ref{e1-2}), we have $2k-8-w_{1}\not\in A$. Moreover,
$$2k-8-w_{1}+a_{k-1}=(2k-7-w_{1})+(2k-5)\in 2^{\wedge}A. $$
By (\ref{e1-19}) we have $2k-8-w_{1}\not\in 2^{\wedge}A$, thus,
$$|\{i, 2k-8-w_{1}-i\}\cap A|\leqslant 1, \ \ i=0, 1, \ldots, k-4-\frac{w_{1}}{2}. $$
By (\ref{e1-2}) and (\ref{e5-2}) we have
$$\frac{w_{2}}{2}+1=|[0, w_{2}]\cap A|=|[0, 2k-8-w_{1}]\cap A|+\frac{w_{1}}{2}+1, $$
thus,
$$|[0, 2k-8-w_{1}]\cap A|=k-3-\frac{w_{1}}{2}, $$
it follows that
$$|\{i, 2k-8-w_{1}-i\}\cap A|=1, \ \ i=0, 1, \ldots, k-4-\frac{w_{1}}{2}. $$
Hence, $k-4-\frac{w_{1}}{2}\in A$.
By (\ref{e1-3}) we have $k-2+\frac{w_{1}}{2}\in A$. So, $2k-6\in 2^{\wedge}A$, which contradicts with $w_{2}=2k-6$.
Hence, $w_{1}+2=\frac{w_{2}}{2}$, which implies that $w_{1}=k-5$, and then
$$A=\left[0, \frac{k-5}{2}\right]\cup \left[k-3, \frac{3k-9}{2}\right]\cup \{2k-5, 2k-4\}, $$
combining with $|2^{\wedge}A|=3k-7$, we have
\begin{equation}\label{e1-4}A=\{0, 1, 2, 6, 7, 8, 9, 13, 14\}. \end{equation}

{\bf Case 2. }$|W|=1$. Let $W=\{w\}$. Then $4\leqslant w\leqslant 2k-6$, $w-1\notin A$ and $2\mid w$.
Combining with $|2^{\wedge}A|=3k-7$, we have
\begin{equation}\label{e1-6}|S(b)\cap 2^{\wedge}A|=1 \text{ for all } b\in [0, a_{k-1}]\backslash(A\cup W). \end{equation}
By Lemma \ref{LL3-6} (1) and (\ref{e1-6}), we have
$$\left[k-2+\frac{w}{2}, a_{k-1}\right]\subseteq A. $$

{\bf Case 2.1. }$4\leqslant w<2k-6$. Then by (\ref{e1-3}) we have $w+1\not\in A$.
By Lemma \ref{LL3-6} (2) and (\ref{e1-6}), we have
$$\left[0, \frac{w}{2}\right]\subseteq A. $$
Thus,
\begin{equation}\label{e1-25}A=\left[0, \frac{w}{2}\right]\cup \left[k-2+\frac{w}{2}, a_{k-1}\right]. \end{equation}

{\bf Case 2.2. }$w=2k-6$. Then $2k-5\in A$.
Let $\alpha$ be an integer such that $[0, \alpha]\subseteq A$ and $\alpha+1\not\in A$. Then
$$2\leqslant \alpha\leqslant k-3. $$

If $\alpha=k-3$, then
\begin{equation}\label{e1-5}A=[0, k-3]\cup \{2k-5, 2k-4\}. \end{equation}

Since $k-3\in A$, we have $\alpha\neq k-4$. If $\alpha=k-5$, then
\begin{equation}\label{e1-27}A=[0, k-5]\cup \{k-3, k-2, 2k-5, 2k-4\}. \end{equation}

If $\alpha\leqslant k-6$, then we shall show that
\begin{equation}\label{e1-29}[\alpha+1, k-4]\cap A=\emptyset. \end{equation}
Suppose that $[\alpha+1, k-4]\cap A\neq\emptyset$.
Let $\beta=\min([\alpha+1, k-4]\cap A)$. Then $\alpha+2\leqslant \beta\leqslant k-4$ and
\begin{equation}\label{e5-3}[\alpha+1, \beta-1]\cap A=\emptyset. \end{equation}
By the definition of $\alpha$ and (\ref{e5-3}), combining with (\ref{e1-2}), we have
\begin{equation}\label{e1-28}[2k-6-\alpha, 2k-6]\cap A=\emptyset, \ \ [2k-5-\beta, 2k-7-\alpha]\subseteq A. \end{equation}
In addition, since $\beta\leqslant k-4$, we have $\beta<2k-5-\beta$, again by (\ref{e1-2}) we have
$$2k-6-\beta\notin A. $$
Clearly,
$$2k-6-\beta+a_{k-1}=(2k-5-\beta)+(2k-5)\in 2^{\wedge}A. $$
By (\ref{e1-6}) we have $2k-6-\beta\not\in 2^{\wedge}A$, which implies that
$$k-1+\alpha-\beta=|[0, 2k-6-\beta]\cap A|\leqslant \left\lfloor\frac{2k-6-\beta}{2}\right\rfloor+1. $$
Thus, $\beta\geqslant 2\alpha+2$. Hence, $2\alpha-1\notin A$. Since $2\alpha-1\in 2^{\wedge}A$, we have $2\alpha-1+a_{k-1}\not\in 2^{\wedge}A$,
it follows that
$$|\{i, 2\alpha-1+a_{k-1}-i\}\cap A|\leqslant 1, \ \ i=2\alpha-1, \ldots, k-3+\alpha, $$
thus,
$$k-1-\alpha=|[2\alpha-1, a_{k-1}]\cap A|\leqslant k-1-\alpha, $$
which implies that
$$|\{i, 2\alpha-1+a_{k-1}-i\}\cap A|=1, \ \ i=2\alpha-1, \ldots, k-3+\alpha. $$
Combining with $2k-6\not\in A$, we have $2\alpha+1\in A$.
On the other hand, since $\beta\geqslant 2\alpha+2$, by the minimality of $\beta$, we have $2\alpha+1\notin A$, a contradiction.

By the definition of $\alpha$ and (\ref{e1-27}), (\ref{e1-2}), (\ref{e1-29}) we have
$$A=[0, \alpha]\cup [k-3, 2k-7-\alpha]\cup \{2k-5, 2k-4\}, $$
where $2\leqslant \alpha\leqslant k-5$. Since $|2^{\wedge}A|=3k-7$, we have
\begin{equation}\label{e1-30}A=\{0, 1, 2, 5, 6, 7, 11, 12\} \text{ or } \{0, 1, 2, 3, 7, 8, 9, 10, 15, 16\}. \end{equation}

By (\ref{e1-4}), (\ref{e1-25}), (\ref{e1-5}) and (\ref{e1-30}), we have completed the proof of necessity.

This completes the proof of Proposition \ref{LL3-3}.
\end{proof}

By Lemma \ref{LL3-4} and Proposition \ref{LL3-3}, it is to obtain the following Propositions \ref{P2}-\ref{P7-2}

\begin{pro}\label{P2}
Let $A$ be a finite set of $k\geqslant 8$ integers such that $a_{k-2}=2k-4$, $a_{k-1}=2k-3$.
If there exists a positive integer $\alpha$ such that
$$a_{s}+\alpha\in 2^{\wedge}(A+\alpha), \ \ s=2, \ldots, k-1, $$
then $|2^{\wedge}A|\geqslant 3k-6$.
\end{pro}

\begin{pro}\label{P2-1}
Let $A$ be a finite set of $k\geqslant 8$ integers such that $a_{k-2}=2k-4$, $a_{k-1}=2k-3$.
If there exists a positive integer $\alpha$ such that
$$a_{2}=2a_{1}, \ \ a_{2}\neq a_{1}+\alpha, \ \ a_{s}+\alpha\in 2^{\wedge}(A+\alpha), \ \ s=3, \ldots, k-1, $$
then $|2^{\wedge}A|\geqslant 3k-7$. Moreover, $|2^{\wedge}A|=3k-7$ if and only if
$$A=\{0, 1, 2, 6, 7, 8, 12, 13\}. $$
In this case, $\alpha=5$.
\end{pro}

\begin{pro}\label{P2-2}
Let $A$ be a finite set of $k\geqslant 8$ integers such that $a_{k-2}=2k-4$, $a_{k-1}=2k-3$.
If $A$ satisfies
$$a_{s}\in 2^{\wedge}(A\backslash\{a_{0}\}), \ \ s=3, \ldots, k-1, $$
then $|2^{\wedge}A|\geqslant 3k-7$. Moreover, $|2^{\wedge}A|=3k-7$ if and only if
\begin{itemize}
  \item $A=\{0, 1\}\cup [k, 2k-3];$
  \item $A=\{0, 1, 5, 6, 7, 8, 12, 13\};$
  \item $A=\{0, 1, 5, 6, 7, 11, 12, 13\};$
  \item $A=\{0, 1, 4, 5, 6, 9, 10, 14, 15\};$
  \item $A=\{0, 1, 6, 7, 8, 9, 13, 14, 15\}. $
\end{itemize}
\end{pro}

\begin{pro}\label{P7-1}
Let $A$ be a finite set of $k\geqslant 3$ integers such that $a_{k-3}=2k-6$, $a_{k-2}=2k-5$ and $2k-4\leqslant a_{k-1}\leqslant 2k-3$.
If there exists a positive integer $\alpha$ such that
$$a_{s}+\alpha\in 2^{\wedge}(A+\alpha), \ \ s=2, \ldots, k-1, $$
then $|2^{\wedge}A|\geqslant 3k-6$.
\end{pro}

\begin{pro}\label{P7}
Let $A$ be a finite set of $k\geqslant 3$ integers such that $a_{k-3}=2k-6$, $a_{k-2}=2k-5$ and $2k-4\leqslant a_{k-1}\leqslant 2k-3$.
If there exists a positive integer $\alpha$ such that
$$a_{2}=2a_{1}, \ \ a_{2}\neq a_{1}+\alpha, \ \ a_{s}+\alpha\in 2^{\wedge}(A+\alpha), \ \ s=3, \ldots, k-1, $$
then $|2^{\wedge}A|\geqslant 3k-7$. Moreover, $|2^{\wedge}A|=3k-7$ if and only if
\begin{itemize}
  \item $A=\{0, 2, 4, 5, 7\}$. In this case, $\alpha=1$ or $3;$
  \item $A=\{0, 1, 2, 6, 7, 8\}$. In this case, $\alpha=5.$
\end{itemize}
\end{pro}

\begin{pro}\label{P7-2}
Let $A$ be a finite set of $k\geqslant 3$ integers such that $a_{k-3}=2k-6$, $a_{k-2}=2k-5$ and $2k-4\leqslant a_{k-1}\leqslant 2k-3$.
If $A$ satisfies
$$a_{s}\in 2^{\wedge}(A\backslash\{a_{0}\}), \ \ s=3, \ldots, k-1, $$
then $|2^{\wedge}A|\geqslant 3k-7$. Moreover, $|2^{\wedge}A|=3k-7$ if and only if
\begin{itemize}
  \item $A=\{0, 2, 3, 5\};$
  \item $A=\{0, 2, 3, 5, 8, 10, 11, 13\};$
  \item $A=\{0, 1, 5, 6, 7, 8, 12, 13, 14\};$
  \item $A=\{0, 1, 6, 7, 8, 9, 13, 14, 15, 16\};$
  \item $A=\{0, 1\}\cup [k-1, 2k-4].$
\end{itemize}
\end{pro}

\section{Maximal element extensions of dense sets}\label{S7}

\begin{pro}\label{P1}
Let $A$ be a set of $k\geqslant 8$ integers such that $a_{k-2}=2k-4$, $a_{k-1}=2k-3$ and $|2^{\wedge}A|=3k-6$.
If there exists a positive integer $\alpha$ such that
$$a_{s}+\alpha\in 2^{\wedge}(A+\alpha), \ \ s=2, \ldots, k-1, $$
then for any integer $n\geqslant 2k$, we have
$$|2^{\wedge}(A\cup \{n\})|\geqslant 3k-3. $$
\end{pro}

\begin{proof}
It is easy to see that
$$\{n+a_{k-2}, n+a_{k-1}\}\subseteq 2^{\wedge}(A\cup \{n\})\backslash 2^{\wedge}A. $$
If $a_{k-3}+n\in 2^{\wedge}(A\cup \{n\})\backslash 2^{\wedge}A$, the result holds.
Assume that $a_{k-3}+n\in 2^{\wedge}A$. Then
\begin{equation}\label{e111}a_{k-3}+n=a_{k-2}+a_{k-1}. \end{equation}
By $n\geqslant 2k$ we have
$$a_{k-3}\leqslant 2k-7. $$
If $a_{k-3}-1\in A$, then $(a_{k-3}-1)+n\in 2^{\wedge}(A\cup \{n\})\backslash 2^{\wedge}A$, the result holds.
If $a_{k-3}-2\in A$, then $(a_{k-3}-2)+n\in 2^{\wedge}(A\cup \{n\})\backslash 2^{\wedge}A$, the result holds.

Assume that $a_{k-3}-2$, $a_{k-3}-1\not\in A$.
By Lemma \ref{L4} we have $a_{k-3}-1\not\in 2^{\wedge}A$, it follows that
$$|\{i, a_{k-3}-1-i\}\cap A|\leqslant 1, \ \ i=0, 1, \ldots, \left\lfloor\frac{a_{k-3}-1}{2}\right\rfloor, $$
thus,
$$k-3=|[0, a_{k-3}-1]\cap A|\leqslant \left\lfloor\frac{a_{k-3}-1}{2}\right\rfloor+1, $$
which implies that $a_{k-3}=2k-7$ and $2k-8\not\in 2^{\wedge}A$. By (\ref{e111}) we have $n=2k$. By Lemma \ref{L3} we have
$$a_{k-4}+n\in 2^{\wedge}(A\cup \{n\})\backslash 2^{\wedge}A, $$
the result holds.

This completes the proof of Proposition \ref{P1}.
\end{proof}

Similar to the proof of Proposition \ref{P1}, combining with Lemmas \ref{L4}, \ref{L5}, we have the following result.

\begin{pro}\label{P6}
Let $A$ be a finite set of $k\geqslant 8$ integers such that $a_{k-2}=2k-4$, $a_{k-1}=2k-3$ and $|2^{\wedge}A|=3k-6$.
If there exists a positive integer $\alpha$ such that
$$a_{2}=2a_{1}, \ \ a_{2}\neq a_{1}+\alpha, \ \ a_{s}+\alpha\in 2^{\wedge}(A+\alpha), \ \ s=3, \ldots, k-1, $$
then for any integer $n\geqslant 2k$, we have
$$|2^{\wedge}(A\cup \{n\})|\geqslant 3k-3. $$
\end{pro}

\begin{pro}\label{P3}
Let $A$ be a set of $k\geqslant 8$ integers such that $a_{k-2}=2k-4$, $a_{k-1}=2k-3$.
If $a_{s}\in 2^{\wedge}(A\backslash\{a_{0}\})$ for all $s=3, \ldots, k-1$, then for any integers $n_{1}\geqslant 2k$, $n_{2}\geqslant 2k+2$ and $n_{1}<n_{2}$, we have
$$|2^{\wedge}(A\cup \{n_{1}, n_{2}\})|\geqslant 3k-1. $$
\end{pro}

\begin{proof}
Clearly,
$$\{a_{k-2}+n_{1}, a_{k-1}+n_{1}, a_{k-1}+n_{2}, n_{1}+n_{2}\}\subseteq 2^{\wedge}(A\cup \{n_{1}, n_{2}\})\backslash 2^{\wedge}A. $$
If $|2^{\wedge}(A\cup \{n_{1}, n_{2}\})\backslash 2^{\wedge}A|=4$, then
$$a_{k-3}+n_{1}=a_{k-2}+a_{k-1}, \ \ a_{k-2}+n_{2}=a_{k-1}+n_{1}, \ \ a_{k-3}+n_{2}= a_{k-2}+n_{1}, $$
thus, $n_{1}=2k-2$, a contradiction. Hence,
$$|2^{\wedge}(A\cup \{n_{1}, n_{2}\})\backslash 2^{\wedge}A|\geqslant 5. $$

By Theorem E we have $|2^{\wedge}A|\geqslant 3k-7$. If $|2^{\wedge}A|\geqslant 3k-6$, then
$$|2^{\wedge}(A\cup \{n_{1}, n_{2}\})|\geqslant 3k-1. $$

Assume $|2^{\wedge}A|=3k-7$. By Proposition \ref{P2-2} we have the following cases:

\begin{itemize}
  \item $A=\{0, 1\}\cup [k, 2k-3]$.
  \item $A=\{0, 1, 5, 6, 7, 8, 12, 13\}$.
  \item $A=\{0, 1, 5, 6, 7, 11, 12, 13\}$.
  \item $A=\{0, 1, 4, 5, 6, 9, 10, 14, 15\}$.
  \item $A=\{0, 1, 6, 7, 8, 9, 13, 14, 15\}$.
\end{itemize}
In all above cases, we have
$$|2^{\wedge}(A\cup \{n_{1}, n_{2}\})|\geqslant 3k-1. $$

This completes the proof of Proposition \ref{P3}.
\end{proof}

\section{Translative-recursive analysis of internal additive generation}\label{S3}

\begin{pro}\label{P12}
Let $A=\{a_{0}, a_{1}, \ldots, a_{k-1}\}$ be a set of integers such that $0=a_{0}<a_{1}<\cdots<a_{k-1}$ and $\gcd(A)=1$. Then

(i) If $a_{k-1}\geqslant 2k-2$ and $a_{k-1}-a_{1}<2k-4$, then $|2^{\wedge}A|\geqslant 3k-7$.

(ii) If $\gcd(A\backslash\{a_{0}\}-a_{1})\geqslant 2$, then $|2^{\wedge}A|\geqslant 3k-6$.

(iii) If $\gcd(A\backslash\{a_{k-1}\})\geqslant 2$, then $|2^{\wedge}A|\geqslant 3k-6$.
\end{pro}

\begin{proof}
(i) If $a_{k-1}-a_{1}<2k-4$, then the reflection set
$$A^{*}=a_{k-1}-A:=\{a_{0}^{*}, a_{1}^{*}, \ldots, a_{k-1}^{*}\}$$
satisfies $a_{k-2}^{*}<2k-4$, $a_{k-1}^{*}\geqslant 2k-2$. By Theorem G, we have
$$|2^{\wedge}A|=|2^{\wedge}A^{*}|\geqslant 3k-7. $$

(ii) Write
$$A_{1}=\{a_{s}: s=1, \ldots, k-1\}, $$
$$A_{2}=\{a_{1}+a_{t}: t=2, \ldots, k-1\}\cup \{a_{m}+a_{k-1}: m=2, \ldots, k-2\}. $$
Put
$$\widetilde{A}:=A\backslash\{a_{0}\}-a_{1}=\{0, a_{2}-a_{1}, \ldots, a_{k-1}-a_{1}\}. $$
Since $\gcd(\widetilde{A})=d\geqslant 2$, we have
$$a_{s}-a_{1}\equiv 0\pmod{d}, \ \ s=2, \ldots, k-1. $$
If there exist $2\leqslant i, j \leqslant k-1$ such that $a_{1}+a_{i}=a_{j}$, then
$$2a_{1}\equiv a_{1}+a_{i}\equiv a_{j}\equiv a_{1}\pmod{d}, $$
thus, $d\mid a_{1}$, it implies that $d\mid \gcd(A)$, which contradicts with $\gcd(A)=1$.
Hence,
$$A_{1}\cap A_{2}=\emptyset, $$
so,
$$|2^{\wedge}A|\geqslant |A_{1}|+|A_{2}|=3k-6. $$

(iii) If $\gcd(A\backslash\{a_{k-1}\})\geqslant 2$, then by $\gcd(A)=1$ we have
$$2^{\wedge}(A\backslash\{a_{k-1}\})\cap (a_{k-1}+A\backslash\{a_{k-1}\})=\emptyset, $$
it follows that
$$|2^{\wedge}A|\geqslant \left|2^{\wedge}(A\backslash\{a_{k-1}\})\right|+|a_{k-1}+A\backslash\{a_{k-1}\}|\geqslant 3k-6. $$

This completes the proof of Proposition \ref{P12}.
\end{proof}

\begin{pro}\label{lem2.13}
Let $A=\{a_{0}, a_{1}, \ldots, a_{k-1}\}$ be a set of $k\geqslant 6$ integers such that $\gcd(A)=1$, $a_0=0$ and $a_1\geqslant 2$ or $a_{1}=1$, $a_{2}\geqslant 3$.
If $a_4=a_2+a_3$ and
\begin{equation}\label{e6-1}a_{s}\in 2^{\wedge}(A\backslash\{a_{0}\}), \ \ s=3, \ldots, k-1, \end{equation}
then $|2^{\wedge}A|\geqslant 3k-7$ or $\left|2^{\wedge}(A\backslash\{a_{0}\}-a_{1})\backslash 2^{\wedge}(A\backslash\{a_{0}, a_{1}\}-a_{1})\right|\geqslant 4$.
\end{pro}

\begin{proof}
By $\gcd(A)=1$ and (\ref{e6-1}) we have
\begin{equation}\label{e6-2}\gcd(a_{1}, a_{2})=1. \end{equation}

Write
$$B=A\backslash\{a_{0}\}-a_{1}:=\{b_{0}, b_{1}, \ldots, b_{k-2}\}, $$
where $b_{i}=a_{i+1}-a_{1}$ for all $i=0, \ldots, k-2$.
Then
$$b_{1}=a_{2}-a_{1}, \ \ b_{2}=a_{2}, \ \ b_{3}=2a_{2}, $$
thus, $b_{3}\neq b_{1}+b_{2}$, it follows that
$$b_{3}\in 2^{\wedge}B\backslash 2^{\wedge}(B\backslash\{b_{0}\}). $$
Clearly,
$$b_{1}, b_{2}\in 2^{\wedge}B\backslash 2^{\wedge}(B\backslash\{b_{0}\}). $$

If $b_{j}=(j-1)a_{2}$ for $j=2, \ldots, k-2$, then
$$B=\left\{a_{2}-a_{1}\right\}\cup [0, k-3]a_{2}. $$
By (\ref{e6-2}) we have
$$|2^{\wedge}(A\backslash\{a_{0}\})|=|2^{\wedge}B|=3k-9, $$
thus,
$$|2^{\wedge}A|=|2^{\wedge}(A\backslash\{a_{0}\})|+|\{a_{1}, a_{2}\}|=3k-7. $$

If there exists $4\leqslant j\leqslant k-2$ such that $b_{j}\neq (j-1)a_{2}$, then
write
$$s=\min\{4\leqslant j\leqslant k-2: b_{j}\neq (j-1)a_{2}\}. $$
Thus,
\begin{equation}\label{e1-11}B\cap [0, b_{s}]=\{0, a_{2}-a_{1}, a_{2}, 2a_{2}, \ldots, (s-2)a_{2}, b_{s}\}, \end{equation}
it follows that
$$a_{t}=a_{1}+(t-2)a_{2}, \ \ t=3, \ldots, s. $$
By (\ref{e6-1}) we have
$$a_{s+1}=a_{i}+a_{j} \text{ for some } 0<i<j<s+1. $$
If $i=2$, then there exists an integer $j<s+1$ such that
$$b_{s}=a_{s+1}-a_{1}=a_{2}+a_{j}-a_{1}=(j-1)a_{2}>(s-2)a_{2}, $$
thus, $j=s$. Hence, $b_{s}=(s-1)a_{2}$, a contradiction.
Hence, $i\neq 2$, it follows that
$$a_{s+1}=\left\{
\begin{array}{ll}2a_{1}+(s-2)a_{2} & \text{ if } i=1,\\
\\
2a_{1}+(i+j-4)a_{2} & \text{ if } i\geqslant 3,
\end{array}\right.$$
so,
\begin{equation}\label{e1-13}b_{s}=a_{s+1}-a_{1}=a_{1}+(s-2)a_{2} \text{ or } a_{1}+(i+j-4)a_{2}. \end{equation}
If $b_{s}\in 2^{\wedge}(B\backslash\{b_{0}\})$, then by (\ref{e1-11}) we have
$$b_{s}=b_{p}+b_{q} \text{ for some } 0<p<q<s, $$
thus,
$$b_{s}=(s-1)a_{2}-a_{1} \text{ or } (p+q-2)a_{2}. $$
Combining with (\ref{e1-13}), we have $a_{2}=2a_{1}$.
By (\ref{e6-2}) we have $a_{1}=1$, $a_{2}=2$, a contradiction.
Thus, $b_{s}\not\in 2^{\wedge}(B\backslash\{b_{0}\})$, it follows that
$$\{b_{1}, b_{2}, b_{3}, b_{s}\}\subseteq 2^{\wedge}B\backslash 2^{\wedge}(B\backslash\{b_{0}\}). $$
Hence,
$$|2^{\wedge}(A\backslash\{a_{0}\}-a_{1})\backslash 2^{\wedge}(A\backslash\{a_{0}, a_{1}\}-a_{1})|\geqslant 4. $$

This completes the proof of Proposition \ref{lem2.13}.
\end{proof}

\begin{pro}\label{lem2.14}
Let $A=\{a_{0}, a_{1}, \ldots, a_{k-1}\}$ be a set of $k\geq 8$ integers such that $\gcd(A)=1$, $a_0=0$, $a_1\geqslant 2$  or $a_{1}=1$, $a_{2}\geqslant 4$, $a_{3}=a_{1}+a_{2}$ and $a_4=a_1+a_3$.
For each of the sets $A$, $A\backslash\{a_{0}\}-a_{1}$, $A\backslash \{a_{0},a_{1}\}-a_{2}$, every element from the fifth onward is expressible as the sum of two distinct nonzero elements of that set.
Then one of the following cases holds:

(i) $a_{6}=3a_{1}+a_{2}$. In this case, $A\cap [a_{1}, a_{6}]=\{2, 3, 5, 7, 8, 9\}$;

(ii) $a_{6}=2a_{1}+2a_{2}$. In this case, $|2^{\wedge}A|\geqslant 3k-7$.
\end{pro}

\begin{proof}
Write
$$B=A\backslash\{a_{0}\}-a_{1}:=\{b_{0}, b_{1}, \ldots, b_{k-2}\}, $$
$$C=A\backslash \{a_{1}, a_{0}\}-a_{2}:=\{c_{0}, c_{1}, \ldots, c_{k-3}\}, $$
where $b_{i}=a_{i+1}-a_{1}$, $c_{j}=a_{j+2}-a_{2}$ for all $i=0, \ldots, k-2$, $j=0, \ldots, k-3$. Then
\begin{equation}\label{e7-1}a_{s}\in 2^{\wedge}(A\backslash\{a_{0}\}), \ \ s=3, \ldots, k-1, \end{equation}\vskip -7mm
\begin{equation}\label{e7-2}b_{t}\in 2^{\wedge}(B\backslash\{b_{0}\}), \ \ t=4, \ldots, k-2, \end{equation}\vskip -5mm
\begin{equation}\label{e7-3}c_{m}\in 2^{\wedge}(C\backslash\{c_{0}\}), \ \ m=4, \ldots, k-3. \end{equation}
By (\ref{e7-1}) and $\gcd(A)=1$ we have
\begin{equation}\label{e7-4}\gcd(a_{1}, a_{2})=1. \end{equation}

Since $a_{3}=a_{1}+a_{2}$ and $a_4=a_1+a_3$, we have $a_{4}=2a_{1}+a_{2}$, thus,
$$b_{1}=a_{2}-a_{1}, \ \ b_{2}=a_{2}, \ \ b_{3}=a_{1}+a_{2}, \ \ c_{1}=a_{1}, \ \ c_{2}=2a_{1}. $$
By (\ref{e7-2}) we have
$$b_{4}\in \{2a_{2}-a_{1}, \ \ 2a_{2}, \ \ a_{1}+2a_{2}\}. $$

If $b_{4}=2a_{2}-a_{1}$, then $a_{5}=2a_{2}$. Combining with (\ref{e7-1}) we have
$a_{5}=a_{1}+a_{4}$, that is, $2a_{2}=3a_{1}+a_{2}$, thus, $a_{2}=3a_{1}$. By (\ref{e7-4}) we have $a_{1}=1$ and $a_{2}=3$, a contradiction.

If $b_{4}=a_{1}+2a_{2}$, then $a_{5}=2a_{1}+2a_{2}$, thus, $c_{3}=2a_{1}+a_{2}$.
By (\ref{e7-3}) we have
$$c_{4}\in \{3a_{1}+a_{2}, \ \ 4a_{1}+a_{2}\}. $$
If $c_{4}=3a_{1}+a_{2}$, then $b_{5}=2a_{1}+2a_{2}$. By (\ref{e7-2}) we have
$2a_{1}+2a_{2}=3a_{2}$, thus, $a_{2}=2a_{1}$. By (\ref{e7-4}) we have $a_{1}=1$ and $a_{2}=2$, a contradiction.
If $c_{4}=4a_{1}+a_{2}$, then $b_{5}=3a_{1}+2a_{2}$. By (\ref{e7-2}) we have
$3a_{1}+2a_{2}=3a_{2}$ or $a_{1}+3a_{2}$, thus, $a_{2}=2a_{1}$ or $3a_{1}$. By (\ref{e7-4}) we have $a_{1}=1$ and $a_{2}=2$ or $3$, a contradiction.

Hence,
$$b_{4}=2a_{2}, \ \ a_{5}=a_{1}+2a_{2}, \ \ c_{3}=a_{1}+a_{2}. $$
By (\ref{e7-3}) we have
$$c_{4}\in \{3a_{1}, \ \ 2a_{1}+a_{2}, \ \ 3a_{1}+a_{2}\}. $$

If $c_{4}=3a_{1}+a_{2}$, then $b_{5}=2a_{1}+2a_{2}$. By (\ref{e7-2}) we have $2a_{1}+2a_{2}=3a_{2}-a_{1}$ or $3a_{2}$,
which implies that $a_{2}=2a_{1}$ or $3a_{1}$. By (\ref{e7-4}) we have $a_{1}=1$ and $a_{2}=2$ or $3$, a contradiction.

If $c_{4}=3a_{1}$, then $b_{5}=2a_{1}+a_{2}$, $a_{6}=3a_{1}+a_{2}$.
By (\ref{e7-2}) we have $2a_{1}+a_{2}=3a_{2}-a_{1}$, which implies that $2a_{2}=3a_{1}$.
By (\ref{e7-4}) we have $a_{1}=2$ and $a_{2}=3$.
Thus,
$$a_{3}=5, \ \ a_{4}=7, \ \ a_{5}=8, \ \ a_{6}=9. $$

If $c_{4}=2a_{1}+a_{2}$, then $b_{5}=a_{1}+2a_{2}$, $a_{6}=2a_{1}+2a_{2}$, which implies that the following figure:

\begin{center}
\begin{tabular}{|c|ccccc|}
  \hline
  & $a_{7}$ & $\overset{-a_{1}}\longrightarrow$ & $b_{6}$ & $\overset{-(a_{2}-a_{1})}\longrightarrow$ &$c_{5}$ \\
  \hline
  (1) & $3a_{1}+2a_{2}$ && $2a_{1}+2a_{2}$ & &$3a_{1}+a_{2}$ \\
  \hline
  (2) & $a_{1}+3a_{2}$ && $3a_{2}$ & &$a_{1}+2a_{2}$ \\
  \hline
  (3) & $2a_{1}+3a_{2}$ && $a_{1}+3a_{2}$ & &$2a_{1}+2a_{2}$\\
  \hline
  (4) & $3a_{1}+3a_{2}$ && $2a_{1}+3a_{2}$ & &$3a_{1}+2a_{2}$\\
  \hline
  (5) & $4a_{1}+3a_{2}$ && $3a_{1}+3a_{2}$ & &$4a_{1}+2a_{2}$\\
  \hline
  (6) & $3a_{1}+4a_{2}$ && $2a_{1}+4a_{2}$ & &$3a_{1}+3a_{2}$\\
  \hline
\end{tabular}
\end{center}
For case (1), combining with $b_{6}\in 2^{\wedge}(B\backslash\{b_{0}\})$, we have $2a_{1}+2a_{2}=3a_{2}-a_{1}$ or $3a_{2}$,
thus, $a_{1}=1$ and $a_{2}=2$ or $3$, a contradiction.
For case (2), combining with $c_{5}\in 2^{\wedge}(C\backslash\{c_{0}\})$, we have $a_{1}+2a_{2}=3a_{1}+a_{2}$ or $4a_{1}+a_{2}$,
thus, $a_{1}=1$ and $a_{2}=2$ or $3$, a contradiction.
For case (3), combining with $c_{5}\in 2^{\wedge}(C\backslash\{c_{0}\})$, we have $2a_{1}+2a_{2}=4a_{1}+a_{2}$,
thus, $a_{1}=1$ and $a_{2}=2$, a contradiction.
For cases (5) and (6), combining with $c_{1}=a_{1}$, $c_{2}=2a_{1}$, $c_{3}=a_{1}+a_{2}$ and $c_{4}=2a_{1}+a_{2}$, we both have $c_{5}\not\in 2^{\wedge}(C\backslash\{c_{0}\})$, a contradiction.

Now we consider case (4). Let $s\geqslant 8$ be an integer such that
$$a_{j}=(j-4)(a_{1}+a_{2}), \ \ 6\leqslant j\leqslant s-1. $$
Noting that
$$A\cap [a_{1}, a_{s-1}]=\{a_{1}, a_{2}, 2a_{1}+a_{2}, a_{1}+2a_{2}\}\cup [1, s-5](a_{1}+a_{2}), $$
thus,
$$B\cap [b_{1}, b_{s-2}]=\{a_{2}-a_{1}, a_{1}+a_{2}, 2a_{2}\}\cup ([0, s-6](a_{1}+a_{2})+a_{2}), $$
$$C\cap [c_{1}, c_{s-3}]=\{2a_{1}, a_{1}+a_{2}\}\cup ([0, s-6](a_{1}+a_{2})+a_{1}). $$
By (\ref{e7-1}) we have
$$a_{s}\in \left\{
\begin{array}{rl}
(s-5)(a_{1}+a_{2})+a_{1}, &(s-5)(a_{1}+a_{2})+a_{2}, \\
(s-4)(a_{1}+a_{2})+a_{1}, &(s-4)(a_{1}+a_{2})+a_{2}, \\
j(a_{1}+a_{2}), & s-4\leqslant j\leqslant 2s-11
\end{array}
\right\}, $$

If $a_{s}=(s-5)(a_{1}+a_{2})+a_{1}$ or $(s-4)(a_{1}+a_{2})+a_{1}$, then
$$b_{s-1}=(s-5)(a_{1}+a_{2}) \text{ or } (s-4)(a_{1}+a_{2}), $$
combining with $b_{s-1}\in 2^{\wedge}(B\backslash\{b_{0}\})$, we have $a_{1}=1$ and $a_{2}=2$, a contradiction.

If $a_{s}=(s-5)(a_{1}+a_{2})+a_{2}$ or $(s-4)(a_{1}+a_{2})+a_{2}$, then
$$c_{s-2}=(s-5)(a_{1}+a_{2}) \text{ or } (s-4)(a_{1}+a_{2}), $$
combining with $c_{s-2}\in 2^{\wedge}(C\backslash\{c_{0}\})$, then $a_{1}=1$ and $a_{2}=2$, a contradiction.

Hence, $a_{s}=j(a_{1}+a_{2})$ for some $s-4\leqslant j\leqslant 2s-11$, it follows that
$$b_{s-1}=(j-1)(a_{1}+a_{2})+a_{2}. $$
By $b_{s-1}\in 2^{\wedge}(B\backslash\{b_{0}\})$, we have
$$b_{s-1}=(a_{1}+a_{2})+((s-6)(a_{1}+a_{2})+a_{2}), $$
which implies that
$$a_{s}=(s-4)(a_{1}+a_{2}). $$

Therefore,
$$a_{j}=(j-4)(a_{1}+a_{2}), \ \ 6\leqslant j\leqslant k-1, $$
it follows that
$$A=\{a_{1}, a_{2}, 2a_{1}+a_{2}, a_{1}+2a_{2}\}\cup [0, k-5](a_{1}+a_{2}). $$
Combining with $k\geqslant 8$ and (\ref{e7-4}), we have
$$|2^{\wedge}A|\geqslant 4k-15\geqslant 3k-7. $$

This completes the proof of Proposition \ref{lem2.14}.
\end{proof}

\begin{pro}\label{P17}
Let $A$ be a set of $k\geqslant 10$ integers such that $\gcd(A)=1$, $a_{k-2}\geqslant 2k-4$, $a_{k-1}\geqslant 2k-2$ and

(i) $a_{k-1}-a_{1}\geqslant 2k-4$, $a_{k-2}-a_{1}<2k-6$;

(ii) $a_{s}\in 2^{\wedge}(A\backslash\{a_{0}\})$ for all $s=3, \ldots, k-1$;

(iii) $a_{t}+a_{k-1}\in 2^{\wedge}(A\backslash\{a_{k-1}\})$ for all $t=0, \ldots, k-4$.

\noindent Then $|2^{\wedge}A|\geqslant 3k-7$.
\end{pro}

\begin{proof}
Write
$$B=A\backslash\{a_{0}\}-a_{1}=\{b_{0}, b_{1}, \ldots, b_{k-2}\}, \text{ where } b_{j}=a_{j+1}-a_{1} \text{ for all } j=0, \ldots, k-2. $$
By (i) we have
\begin{equation}\label{e4-2}b_{k-3}<2k-6, \ \ b_{k-2}\geqslant 2k-4. \end{equation}
By (ii) we have
\begin{equation}\label{e3-2}b_{j}+a_{1}\in 2^{\wedge}(B+a_{1}), \ \ j=2, \ldots, k-2. \end{equation}
By (iii) we have
\begin{equation}\label{e3-7}b_{t}+b_{k-2}\in 2^{\wedge}(B\backslash\{b_{k-2}\}), \ \ t=0, \ldots, k-5. \end{equation}

If $b_{j}<2j$ for all $j=1, \ldots, k-3$, then by (\ref{e3-7}) and Lemma \ref{P4} we have
$$|2^{\wedge}(A\backslash\{a_{0}\})|=|2^{\wedge}B|\geqslant 3k-9, $$
thus,
$$|2^{\wedge}A|\geqslant |2^{\wedge}(A\backslash\{a_{0}\})|+|\{a_{1}, a_{2}\}|\geqslant 3k-7. $$

Write $s=\max\{j\in [2, k-3]: b_{j-1}\geqslant 2j-2\}$. Then $b_{j}<2j$, $j=s, \ldots, k-3$ and $b_{s-1}\geqslant 2s-2$, thus,
$$2s-2\leqslant b_{s-1}<b_{s}<2s, $$
it follows that
\begin{equation}\label{e4-1}b_{s-1}=2s-2, \ \ b_{s}=2s-1. \end{equation}

{\bf Case 1. }$s=k-3$. Then
\begin{equation}\label{e1-32}b_{k-4}=2k-8=2|B\backslash\{b_{k-2}\}|-4, \ \ b_{k-3}=2k-7=2|B\backslash\{b_{k-2}\}|-3. \end{equation}
By (\ref{e3-2}), (\ref{e1-32}) and Proposition \ref{P2} we have
$$\left|2^{\wedge}(B\backslash\{b_{k-2}\})\right|\geqslant 3|B\backslash\{b_{k-2}\}|-6=3k-12. $$

If $\left|2^{\wedge}(B\backslash\{b_{k-2}\})\right|\geqslant 3k-11$, then
\begin{eqnarray*}
\left|2^{\wedge}A\right|
&\geqslant&
\left|2^{\wedge}(A\backslash\{a_{0}\})\right|+|\{a_{1}, a_{2}\}|=\left|2^{\wedge}B\right|+|\{a_{1}, a_{2}\}|\\
&\geqslant&
\left|2^{\wedge}(B\backslash\{b_{k-2}\})\right|+|\{b_{k-4}+b_{k-2}, b_{k-3}+b_{k-2}\}|+|\{a_{1}, a_{2}\}|\\
&\geqslant&
3k-7.
\end{eqnarray*}

Assume that $\left|2^{\wedge}(B\backslash\{b_{k-2}\})\right|=3k-12$.
By (\ref{e4-2}) we have $b_{k-2}\geqslant 2|B\backslash\{b_{k-2}\}|$.
By (\ref{e3-2}), (\ref{e1-32}) and Proposition \ref{P1} we have
$$|2^{\wedge}(A\backslash\{a_{0}\})|=|2^{\wedge}B|\geqslant 3k-9, $$
thus,
$$|2^{\wedge}A|\geqslant |2^{\wedge}(A\backslash\{a_{0}\})|+|\{a_{1}, a_{2}\}|\geqslant 3k-7. $$

{\bf Case 2. }$s\leqslant k-4$.
Define the sets $B_{1}$ and $B_{2}$ by
$$B_{1}=\{b_{0}, b_{1}, \ldots, b_{s-1}, b_{s}, b_{s+1}\}, $$
$$B_{2}=\{b_{s-1}, b_{s}, b_{s+1}, \ldots, b_{k-3}, b_{k-2}\}. $$
Since $b_{s}-b_{s-1}=1$, we have $\gcd(B_{1})=\gcd(B_{2})=1$.
By $b_{s+1}<2s+2$ and (\ref{e4-1}) we have
$$b_{s+1}=2s \text{ or } 2s+1, $$
it follows that
$$b_{s+1}=2|B_{1}|-4 \text{ or } 2|B_{1}|-3. $$
By (\ref{e3-2}) and Proposition \ref{P7-1} we have
\begin{equation}\label{P17-1.1}|2^{\wedge}B_{1}|\geqslant 3|B_{1}|-6=3s. \end{equation}

Define the set $B^{-}_{2}$ by
$$B^{-}_{2}=B_{2}-b_{s-1}=\{0, 1, b_{s+1}-b_{s-1}, \ldots, b_{k-2}-b_{s-1}\}. $$
Noting that
$$\max B^{-}_{2}=b_{k-2}-b_{s-1}\geqslant 2(k-s)-2=2|B^{-}_{2}|-2. $$
For any $j=1, \ldots, k-s-2$, we have
$$b_{s+j-1}-b_{s-1}<2(s+j-1)-(2s-2)=2j. $$
By (\ref{e3-7}) we have
$$(b_{k-2}-b_{s-1})+(b_{t}-b_{s-1})\in 2^{\wedge}(B_{2}^{-}\backslash\{b_{k-2}-b_{s-1}\}), \ \ s-1\leqslant t\leqslant k-5. $$
By Lemma \ref{P4} we have
\begin{equation}\label{P17-2.1}|2^{\wedge}B_{2}|=|2^{\wedge}B^{-}_{2}|\geqslant 3k-3s-6. \end{equation}

Since
$$2^{\wedge}B_{1}\cup2^{\wedge}B_{2}\subseteq 2^{\wedge}B$$
and
$$2^{\wedge}B_{1}\cap2^{\wedge}B_{2}=\{b_{s-1}+b_{s}, b_{s-1}+b_{s+1}, b_{s}+b_{s+1}\}, $$
by (\ref{P17-1.1}) and (\ref{P17-2.1}) we have
$$|2^{\wedge}(A\backslash\{a_{0}\})|=|2^{\wedge}B|\geqslant 3k-9, $$
thus,
$$|2^{\wedge}A|\geqslant |2^{\wedge}(A\backslash\{a_{0}\})|+|\{a_{1}, a_{2}\}|\geqslant 3k-7. $$

This completes the proof of Proposition \ref{P17}.
\end{proof}

\begin{pro}\label{P10}
Let $A$ be a set of $k\geqslant 10$ integers such that $\gcd(A)=1$, $a_{k-2}\geqslant 2k-4$, $a_{k-1}\geqslant 2k-2$ and

(i) $a_{k-1}-a_{1}\geqslant 2k-4$, $a_{k-2}-a_{1}\geqslant 2k-6$, $a_{k-1}-a_{2}\geqslant 2k-6$, $a_{k-2}-a_{2}<2k-8$;

(ii) $a_{s}\in 2^{\wedge}(A\backslash\{a_{0}\})$ for all $s=3, \ldots, k-1$;

(iii) $a_{t}+a_{k-1}\in 2^{\wedge}(A\backslash\{a_{k-1}\})$ for all $t=0, \ldots, k-4$;

(iv) $a_{4}=a_{1}+a_{3}$, $a_{4}\neq a_{3}+a_{2}-a_{1}$;

(v) for every $s\geqslant 5$, there exist $1<i<j<s$ such that $a_{s}=a_{i}+a_{j}-a_{1}$.

\noindent Then $|2^{\wedge}A|\geqslant 3k-7$.
\end{pro}

\begin{proof}
Write
$$B=A\backslash\{a_{0}\}-a_{1}:=\{b_{0}, b_{1}, \ldots, b_{k-2}\}, \text{ where } b_{j}=a_{j+1}-a_{1} \text{ for all } j=0, \ldots, k-2. $$
Since $a_{4}\neq a_{3}+a_{2}-a_{1}$, we have
$$b_{3}\not\in 2^{\wedge}(B\backslash\{b_{0}\}). $$
By (v) we have
$$b_{s-1}=a_{s}-a_{1}=(a_{i}-a_{1})+(a_{j}-a_{1})\in 2^{\wedge}(B\backslash\{b_{0}\}), $$
thus,
\begin{equation}\label{e1-34}b_{s}\in 2^{\wedge}(B\backslash\{b_{0}\}), \ \ s=4, \ldots, k-2. \end{equation}

Write
$$C=B\backslash\{b_{0}\}-b_{1}:=\{c_{0}, c_{1}, \ldots, c_{k-3}\}, \text{ where } c_{j}=b_{j+1}-b_{1} \text{ for all } j=0, \ldots, k-3. $$
Then
\begin{equation}\label{e11-1}c_{2}=2c_{1}, \ \ c_{2}\neq c_{1}+b_{1}. \end{equation}
By (\ref{e1-34}) we have
\begin{equation}\label{e11-2}c_{s}+b_{1}\in 2^{\wedge}(C+b_{1}), \ \ s=3, \ldots, k-3. \end{equation}
By (i) we have
$$c_{k-4}=b_{k-3}-b_{1}=a_{k-2}-a_{2}<2k-8. $$
By (iii) we have
\begin{equation}\label{e3-6}c_{t}+c_{k-3}\in 2^{\wedge}(C\backslash\{c_{k-3}\}), \ \ t=0, \ldots, k-6. \end{equation}

If $c_{j}<2j$ for $j=1, \ldots, k-4$, then by (\ref{e3-6}) and  Lemma \ref{P4} we have
$$|2^{\wedge}(B\backslash\{b_{0}\})|=|2^{\wedge}C|\geqslant 3k-12, $$
thus,
$$|2^{\wedge}(A\backslash\{a_{0}\})|=|2^{\wedge}B|\geqslant |2^{\wedge}(B\backslash\{b_{0}\})|+|\{b_{1}, b_{2}, b_{3}\}|\geqslant 3k-9, $$
it follows that
$$|2^{\wedge}A|\geqslant |2^{\wedge}(A\backslash\{a_{0}\})|+|\{a_{1}, a_{2}\}|\geqslant 3k-7. $$

Write $s=\max\{j\in [2, k-4]: c_{j-1}\geqslant 2j-2\}$. Then
$$c_{j}<2j,  \ \ j=s, \ldots, k-4 \text{ and } c_{s-1}\geqslant 2s-2, $$
thus,
$$2s-2\leqslant c_{s-1}<c_{s}<2s, $$
it follows that
\begin{equation}\label{e4-3}c_{s-1}=2s-2, \ \ c_{s}=2s-1. \end{equation}

{\bf Case 1. }$s=k-4$. Then
\begin{equation}\label{e1-33}c_{k-5}=2k-10=2|C\backslash\{c_{k-3}\}|-4, \ \ c_{k-4}=2k-9=2|C\backslash\{c_{k-3}\}|-3. \end{equation}
By Theorem E we have
$$\left|2^{\wedge}(C\backslash\{c_{k-3}\})\right|\geqslant 3|C\backslash\{c_{k-3}\}|-7=3k-16. $$

If $\left|2^{\wedge}(C\backslash\{c_{k-3}\})\right|\geqslant 3k-14$, then
\begin{eqnarray*}
\left|2^{\wedge}A\right|
&\geqslant&
\left|2^{\wedge}(A\backslash\{a_{0}\})\right|+|\{a_{1}, a_{2}\}|=\left|2^{\wedge}B\right|+|\{a_{1}, a_{2}\}|\\
&\geqslant&
\left|2^{\wedge}(B\backslash\{b_{0}\})\right|+|\{b_{1}, b_{2}, b_{3}\}|+|\{a_{1}, a_{2}\}|=\left|2^{\wedge}C\right|+|\{b_{1}, b_{2}, b_{3}\}|+|\{a_{1}, a_{2}\}|\\
&\geqslant&
\left|2^{\wedge}(C\backslash\{c_{k-3}\})\right|+|\{c_{k-5}+c_{k-3}, c_{k-4}+c_{k-3}\}|+|\{b_{1}, b_{2}, b_{3}\}|+|\{a_{1}, a_{2}\}|\\
&\geqslant&
3k-7.
\end{eqnarray*}

If $\left|2^{\wedge}(C\backslash\{c_{k-3}\})\right|=3k-16$,
then by (\ref{e11-1}), (\ref{e11-2}), (\ref{e1-33}) and Proposition \ref{P2-1} we have
$$C\backslash\{c_{k-3}\}=\{0, 1, 2, 6, 7, 8, 12, 13\}, \ \ b_{1}=5, $$
thus,
$$B=\{0, 5, 6, 7, 11, 12, 13, 17, 18, b_{9}\}, $$
it implies that
$$A\backslash\{a_{0}\}=B+a_{1}=\{a_{1}, 5+a_{1}, 6+a_{1}, 7+a_{1}, 11+a_{1}, 12+a_{1}, 13+a_{1}, 17+a_{1}, 18+a_{1}, a_{10}\}, $$
By (ii) we have $a_{3}\in 2^{\wedge}(A\backslash\{a_{0}\})$, thus, $a_{1}=1$, it follows that
$$A=\{0, 1, 6, 7, 8, 12, 13, 14, 18, 19, a_{10}\}. $$
But, $12\notin 2^{\wedge}(A\backslash\{a_{0}\})$, which contradicts with (ii).

Assume that
$$\left|2^{\wedge}(C\backslash\{c_{k-3}\})\right|=3k-15=3|C\backslash\{c_{k-3}\}|-6. $$
By (i) we have
$$c_{k-3}=a_{k-1}-a_{2}\geqslant 2k-6=2|C\backslash\{c_{k-3}\}|. $$
By (\ref{e11-1}), (\ref{e11-2}), (\ref{e1-33}) and Proposition \ref{P6} we have
$$|2^{\wedge}(B\backslash\{b_{0}\})|=|2^{\wedge}C|\geqslant 3k-12, $$
thus,
\begin{eqnarray*}
|2^{\wedge}A|
&\geqslant&
|2^{\wedge}(A\backslash\{a_{0}\})|+|\{a_{1}, a_{2}\}|=|2^{\wedge}B|+|\{a_{1}, a_{2}\}|\\
&\geqslant&
|2^{\wedge}(B\backslash\{b_{0}\})|+|\{b_{1}, b_{2}, b_{3}\}|+|\{a_{1}, a_{2}\}|\\
&\geqslant&
3k-7.
\end{eqnarray*}

{\bf Case 2. }$s\leqslant k-5$. Define
$$C_{2}=\{c_{s-1}, c_{s}, c_{s+1}, \ldots, c_{k-4}, c_{k-3}\}. $$
Since $c_{s}-c_{s-1}=1$, we have $\gcd(C_{2})=1$. Define
$$C^{-}_{2}=C_{2}-c_{s-1}=\{0, 1, c_{s+1}-c_{s-1}, \ldots, c_{k-3}-c_{s-1}\}. $$
Noting that
$$\max C^{-}_{2}=c_{k-3}-c_{s-1}\geqslant 2(k-1-s)-2=2|C^{-}_{2}|-2. $$
For any $j=1, \ldots, k-s-3$, we have
$$c_{s+j-1}-c_{s-1}<2(s+j-1)-(2s-2)=2j. $$
By (\ref{e3-6}) we have
$$(c_{t}-c_{s-1})+(c_{k-3}-c_{s-1})\in 2^{\wedge}(C_{2}^{-}\backslash\{c_{k-3}-c_{s-1}\}), \ \ s-1\leqslant t\leqslant k-6. $$
By Lemma \ref{P4} we have
\begin{equation}\label{P17-2}|2^{\wedge}C_{2}|=|2^{\wedge}C^{-}_{2}|\geqslant 3k-3s-9. \end{equation}

Define
$$C_{1}=\{c_{0}, c_{1}, \ldots, c_{s-1}, c_{s}, c_{s+1}\}. $$
By $c_{s+1}<2s+2$ and (\ref{e4-3}), we have
$$c_{s+1}=2s \text{ or } 2s+1, $$
it follows that
$$c_{s+1}=2|C_{1}|-4 \text{ or } 2|C_{1}|-3. $$
By Theorem E we have
$$|2^{\wedge}C_{1}|\geqslant 3s-1. $$

If $|2^{\wedge}C_{1}|\geqslant 3s$, then by
$$2^{\wedge}C_{1}\cup2^{\wedge}C_{2}\subseteq 2^{\wedge}C, $$
$$2^{\wedge}C_{1}\cap2^{\wedge}C_{2}=\{c_{s-1}+c_{s}, c_{s-1}+c_{s+1}, c_{s}+c_{s+1}\} $$
and (\ref{P17-2}), we have
$$|2^{\wedge}(B\backslash\{b_{0}\})|=|2^{\wedge}C|\geqslant 3k-12. $$

Assume that $|2^{\wedge}C_{1}|=3s-1$. By (\ref{e11-1}), (\ref{e11-2}), (\ref{e4-3}) and Proposition \ref{P7}, we have
$$C_{1}=\{0, 2, 4, 5, 7\} \text{ or } \{0, 1, 2, 6, 7, 8\}. $$

If $C_{1}=\{0, 1, 2, 6, 7, 8\}$, then $b_{1}=5$, thus,
$$b_{1}=5, \ \ b_{2}=6, \ \ b_{3}=7, \ \ b_{4}=11, \ \ b_{5}=12, \ \ b_{6}=13. $$
By (ii) and the definition of $B$ we have $a_{1}=1$, thus,
$$a_{1}=1, \ \ a_{2}=6, \ \ a_{3}=7, \ \ a_{4}=8, \ \ a_{5}=12, \ \ a_{6}=13, \ \ a_{7}=14. $$
But, $a_{5}\notin 2^{\wedge}(A\backslash\{a_{0}\})$, which contradicts with (ii).

Hence, $C_{1}=\{0, 2, 4, 5, 7\}$, it follows that $s=3$ and $b_{1}=1$ or $3$.
If $b_{1}=3$, then
$$b_{1}=3, \ \ b_{2}=5, \ \ b_{3}=7, \ \ b_{4}=8, \ \ b_{5}=10. $$
By (ii) and the definition of $B$ we have $a_{1}=2$, thus,
$$a_{1}=2, \ \ a_{2}=5, \ \ a_{3}=7, \ \ a_{4}=9, \ \ a_{5}=10, \ \ a_{6}=12. $$
But, $a_{5}\notin 2^{\wedge}(A\backslash\{a_{0}\})$, which contradicts with (ii).
Hence, $b_{1}=1$, it follows that
$$b_{1}=1, \ \ b_{2}=3, \ \ b_{3}=5, \ \ b_{4}=6, \ \ b_{5}=8. $$
By (\ref{e1-34}) we have $b_{6}=9$ or $b_{6}\geqslant 11$, then
$$c_{5}=b_{6}-b_{1}=8 \text{ or } \geqslant 10. $$
Since $c_{j}<2j$ for $j=s, \ldots, k-4$, we have $c_{5}<10$, thus, $c_{5}=8$. Hence,
$$|2^{\wedge}(C_{1}\cup \{c_{5}\})|=12=3s+3. $$
By
$$2^{\wedge}(C_{1}\cup \{c_{5}\})\cup2^{\wedge}C_{2}\subseteq 2^{\wedge}C, $$
$$2^{\wedge}(C_{1}\cup \{c_{5}\})\cap 2^{\wedge}C_{2}=\{c_{i}+c_{j}: 2\leqslant i<j\leqslant 5\}$$
and (\ref{P17-2}) we have
$$|2^{\wedge}(B\backslash\{b_{0}\})|=|2^{\wedge}C|\geqslant |2^{\wedge}(C_{1}\cup \{c_{5}\})|+|2^{\wedge}C_{2}|-6\geqslant 3k-12. $$
Thus,
\begin{eqnarray*}
|2^{\wedge}A|
&\geqslant&
|2^{\wedge}(A\backslash\{a_{0}\})|+|\{a_{1}, a_{2}\}|=|2^{\wedge}B|+|\{a_{1}, a_{2}\}|\\
&\geqslant&
|2^{\wedge}(B\backslash\{b_{0}\})|+|\{b_{1}, b_{2}, b_{3}\}|+|\{a_{1}, a_{2}\}|\\
&\geqslant&
3k-7.
\end{eqnarray*}

This completes the proof of Proposition \ref{P10}.
\end{proof}

\section{Proof of Theorem \ref{T}}\label{S6}
We shall prove Theorem \ref{T} by induction on $k=|A|$. By calculation, we have the following base case:

\begin{BC}
Let $3\leqslant k\leqslant 9$. Let $A=\{a_{0}, a_{1}, \ldots, a_{k-1}\}$ be a set of integers such that $0=a_{0}<a_{1}<\cdots<a_{k-1}$ and $\gcd(A)=1$.
If $a_{k-1}\geqslant 2k-2$, then $|2^{\wedge}A|\geqslant 3k-7$, except for
$$A=\{0, a, b, a+b, 2a+b, a+2b, 2a+2b\}, $$
where $a$ and $b$ are two sufficiently large and coprime positive integers. In this case, we have $|2^{\wedge}A|=3k-8$.
\end{BC}

Assume that Freiman-Lev conjecture is true for all sets $A$ satisfying
$$0=a_{0}<a_{1}<\cdots<a_{l-1},\; \gcd(A)=1, \;a_{l-2}\geqslant 2l-4, \;a_{l-1}\geqslant 2l-2,$$
where $3\leqslant l<k$ and $l\neq 7$.

Next, we shall show that Theorem \ref{T} holds for all sets $A$ with $|A|=k\geqslant 10$ and
$$0=a_{0}<a_{1}<\cdots<a_{k-1},\; \gcd(A)=1, \;a_{k-2}\geqslant 2k-4, \;a_{k-1}\geqslant 2k-2. $$

By the induction hypothesis and Theorem G, we have the following basic fact:

\begin{BF}
Let $A=\{a_{0}, a_{1}, \ldots, a_{l-1}\}$ be a set of integers such that $3\leqslant l<k$, $l\neq 7$ and
$$0=a_{0}<a_{1}<\cdots<a_{l-1}, \ \ \gcd(A)=1, \ \ a_{l-1}\geqslant 2l-2. $$
Then $|2^{\wedge}A|\geqslant 3l-7$.
\end{BF}

By Proposition \ref{P12}, it is sufficient to consider that the set $A$ satisfying
\begin{equation}\label{eqa5.1}\gcd(A)=1, \;a_{k-2}\geqslant 2k-4, \;a_{k-1}\geqslant 2k-2, \ \ a_{k-1}-a_{1}\geqslant 2k-4, \end{equation}
$$\gcd(A\backslash\{a_{0}\}-a_{1})=1, \ \ \gcd(A\backslash\{a_{k-1}\})=1.$$

Assume that $\gcd(A\backslash\{a_{k-1}\})=1$. Since $a_{k-2}\geqslant 2k-4$, applying basic fact to $A\backslash\{a_{k-1}\}$, we have
$$|2^{\wedge}(A\backslash\{a_{k-1}\})|\geqslant 3(k-1)-7=3k-10. $$
It is easy to see that
$$a_{k-2}+a_{k-1}, a_{k-3}+a_{k-1}\in 2^{\wedge}A\backslash2^{\wedge}(A\backslash\{a_{k-1}\}).$$
If there exists an integer $u\in [0, k-4]$ such that
$$a_{u}+a_{k-1}\in 2^{\wedge}A\backslash2^{\wedge}(A\backslash\{a_{k-1}\}), $$
then
$$|2^{\wedge}A|\geqslant |2^{\wedge}(A\backslash\{a_{k-1}\})|+|\{a_{k-2}+a_{k-1}, a_{k-3}+a_{k-1}, a_{u}+a_{k-1}\}|\geqslant 3k-7. $$

Assume that $\gcd(A\backslash\{a_{0}\}-a_{1})=1$. By (\ref{eqa5.1}) we have
$$\max(A\backslash\{a_{0}\}-a_{1})=a_{k-1}-a_{1}\geqslant 2k-4=2|A\backslash\{a_{0}\}-a_{1}|-2, $$
applying basic fact to $A\backslash\{a_{0}\}-a_{1}$, we have
$$|2^{\wedge}(A\backslash\{a_{0}\})|=|2^{\wedge}(A\backslash\{a_{0}\}-a_{1})|\geqslant 3k-10. $$
Clearly,
$$2^{\wedge}(A\backslash\{a_{0}\})\subseteq 2^{\wedge}A, \ \ \{a_{1}, a_{2}\}\subseteq 2^{\wedge}A\backslash 2^{\wedge}(A\backslash\{a_{0}\})\subseteq A. $$
If there exists an integer $s\in [3, k-1]$ such that
$$a_{s}\in 2^{\wedge}A\backslash2^{\wedge}(A\backslash\{a_{0}\}), $$
then
$$|2^{\wedge}A|\geqslant |2^{\wedge}(A\backslash\{a_{0}\})|+|\{a_{1}, a_{2}, a_{s}\}|\geqslant 3k-7. $$

By the above discussion, it is remain to consider
\begin{equation}\label{eqa5.2}a_{u}+a_{k-1}\in 2^{\wedge}(A\backslash\{a_{k-1}\}), \ \ u=0, \ldots, k-4,\end{equation}
\begin{equation}\label{eqa5.4}a_{s}\in 2^{\wedge}(A\backslash\{a_{0}\}), \ \ s=3, \ldots, k-1. \end{equation}

By $\gcd(A)=1$ and (\ref{eqa5.4}) we have
\begin{equation}\label{eqa4.1}\gcd(a_{1}, a_{2})=1. \end{equation}
Moreover, we have the following two statements:

\begin{FactA}
Assume that $A$ satisfies (\ref{eqa5.1}), (\ref{eqa5.2}), (\ref{eqa5.4}). If $a_{1}=1$, $a_{2}<4$, then $|2^{\wedge}A|\geqslant 3k-7$.
\end{FactA}

In fact, by (\ref{eqa5.4}) we have $a_{3}=a_{1}+a_{2}$, $a_{4}=a_{1}+a_{3}$ or $a_{2}+a_{3}$.
Since $a_{1}=1$, $a_{2}<4$, we have $a_{3}<5$, $a_{4}<8$.
Let $t$ be an integer such that
$$a_{t}\geqslant 2t, \ \ a_{j}<2j, \ \ j=1, 2, \ldots, t-1. $$
By (\ref{eqa5.1}) we have
$$5\leqslant t\leqslant k-2. $$

Put $A_{1}=\{a_{0}, a_{1}, \ldots, a_{t}\}$. Then $|A_{1}|=t+1\geqslant 6$.
By Theorem F we have $|2^{\wedge}A_{1}|\geqslant 3(t+1)-7$. 
Moreover, if $|2^{\wedge}A_{1}|=3(t+1)-7$, then by Theorem F we have either
$$A_{1}=\{0,1,3,4,7,10\}, $$
or
$$a_{1}=1, \ \ a_{2}=3, \ \ a_{3}=4, \ \ a_{4}=6. $$
For the latter case, $a_{4}\notin 2^{\wedge}(A\backslash\{a_{0}\})$, which contradicts with (\ref{eqa5.4}). 

It remains to consider $A_{1}=\{0,1,3,4,7,10\}$. Hence, 
$$a_{1}=1, \ \ a_{2}=3, \ \ a_{3}=4, \ \ a_{4}=7=a_{2}+a_{3}. $$
By (\ref{eqa5.4}) and Proposition \ref{lem2.13}, we have
$$|2^{\wedge}A|\geqslant 3k-7$$
or
\begin{equation}\label{P15-special}
\left|
2^{\wedge}(A\backslash\{a_{0}\}-a_{1})
\backslash
2^{\wedge}(A\backslash\{a_{0},a_{1}\}-a_{1})
\right|
\geqslant4.
\end{equation}
If $|2^{\wedge}A|\geqslant3k-7$, then Fact A follows.

Assume that (\ref{P15-special}) holds. Put
$$B=A\backslash\{a_{0}\}-a_{1}, \ \ C=A\backslash\{a_{0}, a_{1}\}-a_{2}. $$
Since $a_{1}=1$, $a_{2}=3$ and $a_{3}=4$, we have $\{0, 2, 3\}\subseteq B$, thus,
$$\gcd(B)=1. $$

If $\gcd(C)>1$, then $C=B\backslash\{0\}-2$. By Proposition \ref{P12} (ii), we have
$$|2^{\wedge}B|\geqslant3|B|-6=3k-9.$$

If $\gcd(C)=1$, then $|C|=k-2$ and by (\ref{eqa5.1}),
$$\max C=a_{k-1}-a_{2}=a_{k-1}-3\geqslant2k-5>2k-6=2|C|-2.$$
By basic fact, we have
$$|2^{\wedge}C|\geqslant3|C|-7=3k-13.$$
By (\ref{P15-special}), we have
$$|2^{\wedge}B|
\geqslant |2^{\wedge}C|+4
\geqslant3k-9.$$

Thus, in both cases,
$$|2^{\wedge}(A\backslash\{a_{0}\})|
=|2^{\wedge}B|
\geqslant3k-9.$$
Since $a_{1}=1$, $a_{2}=3$ and $\{1,3\}\subseteq 2^{\wedge}A\backslash 2^{\wedge}(A\backslash\{a_{0}\})$, 
we have
$$|2^{\wedge}A|\geqslant3k-9+2=3k-7. $$
Therefore, Fact A follows in this case.

Hence, for the remaining case,
\begin{equation}\label{P15-1}
|2^{\wedge}A_{1}|\geqslant3(t+1)-6=3t-3.
\end{equation}

If $t\neq k-5$, then put $A_{2}=\{a_{t-2}, a_{t-1}, \ldots, a_{k-1}\}$, thus, $|A_{2}|=k-t+2\geqslant 4$ and $|A_{2}|\neq 7$.
Let $A^{-}_{2}=A_{2}-a_{t-2}$. Assume that $\gcd(A^{-}_{2})=d>1$. Then
$$a_{j}\equiv a_{t-2}\pmod{d}, \ \ j=t-2, \ldots, k-1. $$
If $a_{i}\equiv a_{t-2}\pmod{d}$ for all $i=0, \ldots, t-3$, then $d\mid \gcd(A)$, a contradiction.
Let $a_{m}$ be the largest integer of $A$ such that $a_{m}\not \equiv a_{t-2}\pmod{d}$.
Then $0\leqslant m\leqslant t-3\leqslant k-5$ and
$$a_{m}+a_{k-1}\notin 2^{\wedge}(A\backslash\{a_{k-1}\}), $$
which contradicts with (\ref{eqa5.2}). Thus,
$$\gcd(A^{-}_{2})=1. $$
Since
$$\max A^{-}_{2}=a_{k-1}-a_{t-2}>2k-2t+2=2|A^{-}_{2}|-2, $$
applying basic fact to $A^{-}_{2}$, we have
\begin{equation}\label{P15-2}|2^{\wedge}A_{2}|=|2^{\wedge}A^{-}_{2}|\geqslant 3|A^{-}_{2}|-7=3k-3t-1. \end{equation}

In addition,
$$2^{\wedge}A_{1}\cap 2^{\wedge}A_{2}=\{a_{t-2}+a_{t-1}, a_{t-2}+a_{t}, a_{t-1}+a_{t}\}. $$
By (\ref{P15-1}) and (\ref{P15-2}) we have
$$|2^{\wedge}A|\geqslant |2^{\wedge}A_{1}|+|2^{\wedge}A_{2}|-3=3k-7. $$

If $t=k-5$, then put $B_{2}=\{a_{k-8}, \ldots, a_{k-1}\}$, thus, $|B_{2}|=8$.
Let $B^{-}_{2}=B_{2}-a_{k-8}$. Similarly, we have $\gcd(B^{-}_{2})=1$. Noting that
$$\max(B^{-}_{2})=a_{k-1}-a_{k-8}>14, $$
by Base case we have
$$|2^{\wedge}B_{2}|=|2^{\wedge}B^{-}_{2}|\geqslant 3|B^{-}_{2}|-7=17. $$
In addition,
$$2^{\wedge}A_{1}\cap 2^{\wedge}B_{2}=\{a_{i}+a_{j}: k-8\leqslant i<j\leqslant k-5\}. $$
Thus,
$$|2^{\wedge}A|\geqslant |2^{\wedge}A_{1}|+|2^{\wedge}B_{2}|-6=3k-7. $$

Noting that if $A$ satisfies that (\ref{eqa5.1}), (\ref{eqa5.2}), (\ref{eqa5.4}), then the reflection set $A^{*}$ satisfies
\begin{equation}\label{facta}\gcd(A^{*})=1, a_{k-2}^{*}\geqslant 2k-4, a_{k-1}^{*}\geqslant 2k-2, a_{k-1}^{*}-a_{1}^{*}\geqslant 2k-4\end{equation}
\begin{equation}\label{factc}a_{t}^{*}+a_{k-1}^{*}\in 2^{\wedge}(A^{*}\backslash\{a_{k-1}^{*}\}), \ \  t=0, \ldots, k-4, \end{equation}
\begin{equation}\label{factd} a_{j}^{*}\in 2^{\wedge}(A^{*}\backslash\{a_{0}^{*}\}), \ \ j=3, \ldots, k-1. \end{equation}

Similar to the proof of Fact A, we have

\begin{FactB}
Assume that $A^{*}$ satisfies (\ref{facta})-(\ref{factd}). If $a_{1}^{*}=1$, $a_{2}^{*}<4$, then $|2^{\wedge}A|=|2^{\wedge}A^{*}|\geqslant 3k-7$.
\end{FactB}

Assume that set $A$ satisfies (\ref{eqa5.1}), (\ref{eqa5.2}), (\ref{eqa5.4}). By Proposition \ref{P17}, it is sufficient to consider the set $A$ satisfying
\begin{equation}\label{eqa5.23}a_{k-2}-a_{1}\geqslant 2k-6. \end{equation}

Write
$$B=A\backslash\{a_{0}\}-a_{1}=\{0, a_{2}-a_{1}, \ldots, a_{k-1}-a_{1}\}:=\{b_{0}, b_{1}, \ldots, b_{k-3}, b_{k-2}\}, $$
$$C=B\backslash \{b_{0}\}-b_{1}=A\backslash \{a_{0}, a_{1}\}-a_{2}=\{c_{0}, c_{1}, \ldots, c_{k-4}, c_{k-3}\}, $$
$$D=C\backslash\{c_{0}\}-c_{1}=A\backslash \{a_{0}, a_{1}, a_{2}\}-a_{3}. $$

To complete the proof of Theorem \ref{T}, we will establish the relationship between $|2^{\wedge}A|$ and $|2^{\wedge}B|$, $|2^{\wedge}C|$, $|2^{\wedge}D|$.
The proof is accomplished with the following four claims.

\begin{ClaimA}Assume that set $A$ satisfies (\ref{eqa5.1}), (\ref{eqa5.2}), (\ref{eqa5.4}) and (\ref{eqa5.23}).
If $a_{k-1}-a_{2}<2k-6$, then $|2^{\wedge}A|\geqslant 3k-7$.
\end{ClaimA}

Since $a_{k-1}-a_{2}<2k-6$, we have  $a_{k-3}^{*}<2k-6$.
If $a_{j}^{*}<2j$ for $j=1, \ldots, k-3$, then $a_{1}^{*}=1$, $a_{2}^{*}<4$.
By Fact B we have $|2^{\wedge}A|=|2^{\wedge}A^{*}|\geqslant 3k-7$.

Let $s\in [2, k-3]$ be an integer such that $a_{s-1}^{*}\geqslant 2(s-1)$ and $a_{j}^{*}<2j$, $j=s, \ldots, k-3$.
Then
$2s-2\leqslant a_{s-1}^{*}<a_{s}^{*}<2s, $
thus,
\begin{equation}\label{eqa4.8}a_{s-1}^{*}=2s-2, \ \ a_{s}^{*}=2s-1. \end{equation}

If $s=k-3$, then
\begin{equation}\label{eqa4.11}a_{k-4}^{*}=2k-8, \ \ a_{k-3}^{*}=2k-7, \end{equation}
By Proposition \ref{P3} we have
$$|2^{\wedge}A|=|2^{\wedge}A^{*}|\geqslant 3k-7. $$

Now, we assume that $s<k-3$. Then $a_{s+1}^{*}<2s+2$. Combining with (\ref{eqa4.8}), we have
\begin{equation}\label{eqa4.10}a_{s+1}^{*}=2s \text{ or } 2s+1. \end{equation}
Write $$\widetilde{A}_{1}=\{a_{0}^{*}, a_{1}^{*}, \ldots, a_{s-1}^{*}, a_{s}^{*}, a_{s+1}^{*}\}, \ \ \widetilde{A}_{2}=\{a_{s-1}^{*}, a_{s}^{*}, a_{s+1}^{*}, \ldots, a_{k-1}^{*}\}. $$
By Theorem E we have
$$|2^{\wedge}\widetilde{A}_{1}|\geqslant 3(s+2)-7=3s-1. $$

Define the set $\widetilde{A}^{-}_{2}$ by
$$\widetilde{A}^{-}_{2}=\widetilde{A}_{2}-a_{s-1}^{*}=\{0, 1, a_{s+1}^{*}-a_{s-1}^{*}, \ldots, a_{k-1}^{*}-a_{s-1}^{*}\}. $$
Then
\begin{equation}\label{eqa5.6}\max\widetilde{A}^{-}_{2}=a_{k-1}^{*}-a_{s-1}^{*}\geqslant 2(k-s+1)-2=2|\widetilde{A}^{-}_{2}|-2. \end{equation}
Noting that the third smallest integer of $\widetilde{A}^{-}_{2}$ is either 2 or 3.
If $|\widetilde{A}^{-}_{2}|=7$ and $|2^{\wedge}\widetilde{A}^{-}_{2}|<14$, then by Base case we have the largest integer of $\widetilde{A}^{-}_{2}$ is not greater than 8, which contradicts with (\ref{eqa5.6}).
Hence, if $|\widetilde{A}^{-}_{2}|=7$, then
\begin{equation}\label{e2-2-3}|2^{\wedge}\widetilde{A}^{-}_{2}|\geqslant 14. \end{equation}
Assume that $|\widetilde{A}^{-}_{2}|\neq 7$. By (\ref{eqa5.6}), applying basic fact to $\widetilde{A}^{-}_{2}$, we have
\begin{equation}\label{e2-2}|2^{\wedge}\widetilde{A}_{2}|=|2^{\wedge}\widetilde{A}^{-}_{2}|\geqslant 3k-3s-4. \end{equation}

If $|2^{\wedge}\widetilde{A}_{1}|\geqslant 3|\widetilde{A}_{1}|-6=3s$, then by
$$2^{\wedge}\widetilde{A}_{1}\cup2^{\wedge}\widetilde{A}_{2}\subseteq 2^{\wedge}A^{*}$$
and
$$2^{\wedge}\widetilde{A}_{1}\cap2^{\wedge}\widetilde{A}_{2}=\{a_{s-1}^{*}+a_{s}^{*}, a_{s-1}^{*}+a_{s+1}^{*}, a_{s}^{*}+a_{s+1}^{*}\}, $$
combining with (\ref{e2-2-3}), (\ref{e2-2}) we have
$$|2^{\wedge}A|=|2^{\wedge}A^{*}|\geqslant |2^{\wedge}\widetilde{A}_{1}|+|2^{\wedge}\widetilde{A}_{2}|-3\geqslant 3k-7. $$

If $|2^{\wedge}\widetilde{A}_{1}|=3|\widetilde{A}_{1}|-7=3s-1$, then by (\ref{factd}), (\ref{eqa4.8}), (\ref{eqa4.10}) and Proposition \ref{P7-2} we need consider the following cases:

\begin{itemize}
  \item $\widetilde{A}_{1}=\{0, 2, 3, 5\}$. Then $s=2$. Since $a_{j}^{*}<2j$, $j=2, \ldots, k-3$,
  combining with (\ref{eqa5.23}) and the fact that $a_{1}^{*}=2$, we have
  $$a_{j}^{*}-a_{1}^{*}<2(j-1), \ \ j=2, \ldots, k-3, \ \ a_{k-2}^{*}-a_{1}^{*}\geqslant 2k-6. $$
  Moreover, for all $t=1, \ldots, k-4$, by (\ref{factc}) we have
  $$(a_{k-1}^{*}-a_{1}^{*})+(a_{t}^{*}-a_{1}^{*})\in 2^{\wedge}(\widetilde{A}^{-}_{2}\backslash\{a_{k-1}^{*}-a_{1}^{*}\}). $$
  By (\ref{factd}) we have $a_{4}^{*}\geqslant 7$, thus, $a_{4}^{*}=7$.
  Hence,
  $$1, 3, 5\in \widetilde{A}^{-}_{2}\backslash\{a_{k-1}^{*}-a_{1}^{*}\}. $$
  By Proposition \ref{P16} and Proposition \ref{P16.1} we have
  $$|2^{\wedge}(\widetilde{A}^{-}_{2}\backslash\{a_{k-1}^{*}-a_{1}^{*}\})|\geqslant 3(k-2)-5=3k-11, $$
  thus,
  \begin{eqnarray*}
  |2^{\wedge}\widetilde{A}_{2}|
  &=&
  |2^{\wedge}\widetilde{A}^{-}_{2}|\geqslant |2^{\wedge}(\widetilde{A}_{2}\backslash\{a_{k-1}^{*}\})|+|\{a_{k-3}^{*}+a_{k-1}^{*}, a_{k-2}^{*}+a_{k-1}^{*}\}|\\
  &\geqslant&
  |2^{\wedge}(\widetilde{A}^{-}_{2}\backslash\{a_{k-1}^{*}-a_{1}^{*}\})|+2\geqslant 3k-9,
  \end{eqnarray*}
  it follows that
  $$|2^{\wedge}A|=|2^{\wedge}A^{*}|\geqslant |2^{\wedge}\widetilde{A}_{2}|+|\{a_{1}^{*}, a_{2}^{*}\}|\geqslant 3k-7. $$
  \item $\widetilde{A}_{1}=\{0, 2, 3, 5, 8, 10, 11, 13\}$. Then $s=6$. By (\ref{factd}) we have $a_{8}^{*}\geqslant 14$.
  Since $a_{j}^{*}<2j$, $j=6, \ldots, k-3$, we have $a_{8}^{*}<16$, thus, $a_{8}^{*}=14$ or $15$. Hence,
  $$|2^{\wedge}(\widetilde{A}_{1}\cup \{a_{8}^{*}\})|\geqslant 21. $$
  Since
  $$2^{\wedge}(\widetilde{A}_{1}\cup \{a_{8}^{*}\})\cup2^{\wedge}\widetilde{A}_{2}\subseteq 2^{\wedge}A^{*}, $$
  $$2^{\wedge}(\widetilde{A}_{1}\cup \{a_{8}^{*}\})\cap 2^{\wedge}\widetilde{A}_{2}=\{a_{i}^{*}+a_{j}^{*}: 5\leqslant i<j\leqslant 8\}. $$
  By (\ref{e2-2-3}), (\ref{e2-2}) we have
  $$|2^{\wedge}A|=|2^{\wedge}A^{*}|\geqslant |2^{\wedge}(\widetilde{A}_{1}\cup \{a_{8}^{*}\})|+|2^{\wedge}\widetilde{A}_{2}|-6\geqslant 3k-7. $$
  \item Similarly, we can show that the result is true for
  $$s=7, \;\widetilde{A}_{1}=\{0, 1, 5, 6, 7, 8, 12, 13, 14\}, $$
  $$s=8, \;\widetilde{A}_{1}=\{0, 1, 6, 7, 8, 9, 13, 14, 15, 16\}.$$
  \item $\widetilde{A}_{1}=\{0, 1\}\cup [s+1, 2s]$. Then $s\geqslant 3$. If $s=3$, then $\widetilde{A}_{1}=\{0, 1, 4, 5, 6\}$. Thus,
  $$a_{k-2}^{*}-a_{2}^{*}\geqslant 2k-8, \ \ a_{j}^{*}-a_{2}^{*}<2(j-2), \ \ j=3, \ldots, k-3. $$
  By $0, 1, 2\in \widetilde{A}^{-}_{2}\backslash\{a_{k-1}^{*}-a_{2}^{*}\}$ and Theorem F we have
  $$|2^{\wedge}(\widetilde{A}_{2}\backslash\{a_{k-1}^{*}\})|=|2^{\wedge}(\widetilde{A}^{-}_{2}\backslash\{a_{k-1}^{*}-a_{2}^{*}\})|\geqslant 3k-15, $$
  thus,
  $$|2^{\wedge}\widetilde{A}_{2}|\geqslant |2^{\wedge}(\widetilde{A}_{2}\backslash\{a_{k-1}^{*}\})|+|\{a_{k-3}^{*}+a_{k-1}^{*}, a_{k-2}^{*}+a_{k-1}^{*}\}|\geqslant 3k-13. $$
  Noting that
  $$\{1, 4, 5, 6, 7\}\subseteq 2^{\wedge}A^{*}\backslash 2^{\wedge}\widetilde{A}_{2}, $$
  it follows that
  $$|2^{\wedge}A^{*}\backslash 2^{\wedge}\widetilde{A}_{2}|\geqslant 5. $$
  If $|2^{\wedge}A^{*}\backslash 2^{\wedge}\widetilde{A}_{2}|\geqslant 6$, then
  $$|2^{\wedge}A|=|2^{\wedge}A^{*}|\geqslant 3k-7. $$
  Assume that $|2^{\wedge}A^{*}\backslash 2^{\wedge}\widetilde{A}_{2}|=5$.
  If $a_{5}^{*}=7$ or $8$, then $8\in 2^{\wedge}A^{*}\backslash 2^{\wedge}\widetilde{A}_{2}$, a contradiction.
  Since $a_{5}^{*}<10$, we have $a_{5}^{*}=9$. Similarly, we have $a_{6}^{*}=10$, $a_{7}^{*}=13. $
  Thus,
  $$1, 2, 5, 6, 9\in \widetilde{A}^{-}_{2}\backslash\{a_{k-1}^{*}-a_{2}^{*}\}. $$
  By (\ref{factc}) we have
  $$(a_{k-1}^{*}-a_{2}^{*})+(a_{j}^{*}-a_{2}^{*})\in 2^{\wedge}(\widetilde{A}^{-}_{2}\backslash\{a_{k-1}^{*}-a_{2}^{*}\}) $$
  for all $j=2, \ldots, k-4$.
  By Proposition \ref{P16.3} we have
  $$|2^{\wedge}(\widetilde{A}^{-}_{2}\backslash\{a_{k-1}^{*}-a_{2}^{*}\})|\geqslant 3(k-3)-5=3k-14, $$
  thus,
  \begin{eqnarray*}
  |2^{\wedge}A|&=&|2^{\wedge}A^{*}|\geqslant 5+|2^{\wedge}(\widetilde{A}_{2}\backslash\{a_{k-1}^{*}\})|+|\{a_{k-3}^{*}+a_{k-1}^{*}, a_{k-2}^{*}+a_{k-1}^{*}\}|\\
  &=& |2^{\wedge}(\widetilde{A}^{-}_{2}\backslash\{a_{k-1}^{*}-a_{2}^{*}\})|+7\geqslant 3k-7.
  \end{eqnarray*}
  Now we assume that $s\geqslant 4$. By Theorem F we have
  $$|2^{\wedge}(\widetilde{A}^{-}_{2}\backslash\{a_{k-1}^{*}-a_{s-1}^{*}\})|\geqslant 3(k-s)-6. $$
  If $s=k-4$, then
  $$A^{*}=\{0, 1\}\cup [k-3, 2k-8]\cup \{a_{k-2}^{*}, a_{k-1}^{*}\}, $$
  thus,
  $$|2^{\wedge}A|=|2^{\wedge}A^{*}|\geqslant 3k-7. $$
  Assume that $s<k-4$. Then $a_{s+2}^{*}<2s+4$. If $a_{s+2}^{*}=2s+1$, then
  \begin{eqnarray*}|2^{\wedge}A|&=&|2^{\wedge}A^{*}|\geqslant |2^{\wedge}\widetilde{A}_{1}|+|2^{\wedge}(\widetilde{A}_{2}\backslash\{a_{k-1}^{*}\})|\\
  &&+|\{a_{k-3}^{*}+a_{k-1}^{*}, a_{k-2}^{*}+a_{k-1}^{*}\}|+|\{2s+2\}|-3\\
  &\geqslant& 3k-7. \end{eqnarray*}
  By (\ref{factd}) we have $a_{s+2}^{*}\neq 2s+2$. If $a_{s+2}^{*}=2s+3$, then
  \begin{eqnarray*}|2^{\wedge}A|&=&|2^{\wedge}A^{*}|\geqslant |2^{\wedge}(\widetilde{A}_{1}\cup \{a_{s+2}^{*}\})|+|2^{\wedge}(\widetilde{A}_{2}\backslash\{a_{k-1}^{*}\})|\\
  &&+|\{a_{k-3}^{*}+a_{k-1}^{*}, a_{k-2}^{*}+a_{k-1}^{*}\}|-6\\
  &\geqslant& 3k-7. \end{eqnarray*}
\end{itemize}

By Claim 1, it is sufficient to consider set $A$ satisfies (\ref{eqa5.1}), (\ref{eqa5.2}), (\ref{eqa5.4}), (\ref{eqa5.23}) and
\begin{equation}\label{eqa5.17}a_{k-1}-a_{2}\geqslant 2k-6. \end{equation}
By (\ref{eqa5.4}) we have
$$a_{3}=a_{1}+a_{2}, \ \ a_{4}=a_{1}+a_{3} \text{ or } a_{2}+a_{3}. $$

We have the following claim.
\begin{ClaimB} Assume that set $A$ satisfies (\ref{eqa5.1}), (\ref{eqa5.2}), (\ref{eqa5.4}), (\ref{eqa5.23}), (\ref{eqa5.17}). If one of the following five cases holds:

(i) $\gcd(C)\geqslant 2$;

(ii) $\gcd(C)=1$ and $a_{4}=a_{2}+a_{3}$;

(iii) $\gcd(C)=1$, $a_{4}=a_{1}+a_{3}$ and $b_{j}\notin 2^{\wedge}(B\backslash \{b_{0}\})$ for some $4\leqslant j\leqslant k-2$;

(iv) $\gcd(C)=1$, $a_{4}=a_{1}+a_{3}$, $b_{j}\in 2^{\wedge}(B\backslash \{b_{0}\})$ for all $j=4, \ldots, k-2$ and $|2^{\wedge}C|\geqslant 3k-12$;

(v) $\gcd(C)=1$, $a_{4}=a_{1}+a_{3}$, $b_{j}\in 2^{\wedge}(B\backslash \{b_{0}\})$ for all $j=4, \ldots, k-2$ and $a_{k-2}-a_{2}<2k-8$,

\noindent then we have $|2^{\wedge}A|\geqslant 3k-7$.
\end{ClaimB}

To prove (i).  Since $\gcd(C)\geqslant 2$, by Proposition \ref{P12} we have
$$|2^{\wedge}B|\geqslant 3|B|-6=3k-9, $$
thus,
$$|2^{\wedge}A|\geqslant |2^{\wedge}(A\backslash \{a_{0}\})|+|\{a_{1}, a_{2}\}|=|2^{\wedge}B|+|\{a_{1}, a_{2}\}|\geqslant 3k-7. $$

Assume that $\gcd(C)=1$. By (\ref{eqa5.17}) we have
$$\max C=a_{k-1}-a_{2}\geqslant 2k-6=2|C|-2, $$
applying basic fact to $C$, we have
\begin{equation}\label{e1-31}|2^{\wedge}(B\backslash \{b_{0}\})|=|2^{\wedge}C|\geqslant 3|C|-7=3k-13. \end{equation}

To prove (ii). If $a_{1}=1$ and $a_{2}<4$, then Fact A gives $|2^{\wedge}A|\geqslant 3k-7$.
Throughout the rest of the proof we may assume that $a_{1}\geqslant 2$ or $a_{1}=1$, $a_{2}\geqslant 4$.
Noting that $a_{4}=a_{2}+a_{3}$, by Proposition \ref{lem2.13}, it is sufficient to consider
$$|2^{\wedge}(A\backslash\{a_{0}\}-a_{1})\backslash 2^{\wedge}(A\backslash\{a_{0}, a_{1}\}-a_{1})|\geqslant 4, $$
that is
$$|2^{\wedge}B\backslash 2^{\wedge}(B\backslash\{b_{0}\})|\geqslant 4. $$
By (\ref{e1-31}) we have
$$|2^{\wedge}(A\backslash\{a_{0}\})|\geqslant |2^{\wedge}(B\backslash\{b_{0}\})|+4=3k-9, $$
thus,
$$|2^{\wedge}A|\geqslant |2^{\wedge}(A\backslash\{a_{0}\})|+|\{a_{1}, a_{2}\}|\geqslant 3k-7. $$

To prove (iii). Noting that
\begin{equation}\label{eqa5.18}a_{4}=a_{1}+a_{3},\end{equation}
we have
$$b_{1}=a_{2}-a_{1}, \ \ b_{2}=a_{2}, \ \ b_{3}=a_{1}+a_{2}. $$

If $b_{3}=b_{1}+b_{2}$, then $a_{1}+a_{2}=2a_{2}-a_{1}$, thus, $a_{2}=2a_{1}$.
By (\ref{eqa4.1}) we have  $a_{1}=1$, $a_{2}=2$. By Fact A we have $|2^{\wedge}A|\geqslant 3k-7$.

If $b_{3}\neq b_{1}+b_{2}$, then
$$b_{3}\in 2^{\wedge}B\backslash 2^{\wedge}(B\backslash \{b_{0}\})$$
and
\begin{equation}\label{eqa5.9}a_{4}\neq a_{3}+a_{2}-a_{1}. \end{equation}
If there exists an integer $4\leqslant j\leqslant k-2$ such that $b_{j}\in 2^{\wedge}B\backslash 2^{\wedge}(B\backslash \{b_{0}\})$,
then
$$|2^{\wedge}(A\backslash \{a_{0}\})|=|2^{\wedge}B|\geqslant |2^{\wedge}(B\backslash \{b_{0}\})|+|\{b_{1}, b_{2}, b_{3}, b_{j}\}|\geqslant 3k-9, $$
thus,
$$|2^{\wedge}A|\geqslant |2^{\wedge}(A\backslash \{a_{0}\})|+|\{a_{1}, a_{2}\}|\geqslant 3k-7. $$

Noting that
\begin{equation}\label{eqa5.14}b_{j}\in 2^{\wedge}(B\backslash \{b_{0}\}), \ \ j=4, \ldots, k-2. \end{equation}
Then for every $s\geqslant 5$, we have
\begin{equation}\label{eqa5.7}a_{s}=a_{u}+a_{v}-a_{1} \text{ for some } 1<u<v<s. \end{equation}

To prove (iv). Since $|2^{\wedge}(B\backslash \{b_{0}\})|=|2^{\wedge}C|\geqslant 3k-12$, we have
$$|2^{\wedge}(A\backslash \{a_{0}\})|=|2^{\wedge}B|\geqslant |2^{\wedge}(B\backslash \{b_{0}\})|+|\{b_{1}, b_{2}, b_{3}\}|\geqslant 3k-9, $$
thus,
$$|2^{\wedge}A|\geqslant |2^{\wedge}(A\backslash \{a_{0}\})|+|\{a_{1}, a_{2}\}|\geqslant 3k-7. $$

To prove (v). Since $a_{k-2}-a_{2}<2k-8$, by Proposition \ref{P10}, we have $|2^{\wedge}A|\geqslant 3k-7$.

By Claim 2, it is sufficient to consider \begin{equation}\label{eqa5.8}|2^{\wedge}C|=3k-13, \end{equation}\vskip-5mm
\begin{equation}\label{eqa5.19}a_{k-2}-a_{2}\geqslant 2k-8,\end{equation}

Now, we shall prove the following claim:

\begin{ClaimC}Assume that set $A$ satisfies
(\ref{eqa5.1}), (\ref{eqa5.2}), (\ref{eqa5.4}), (\ref{eqa5.23}), (\ref{eqa5.17}), (\ref{eqa5.18}), (\ref{eqa5.9}), (\ref{eqa5.7}), (\ref{eqa5.8}), (\ref{eqa5.19}).
If $a_{k-1}-a_{3}\geqslant 2k-8$, then $|2^{\wedge}A|\geqslant 3k-7$.
\end{ClaimC}

Write $c_{i}=a_{i+2}-a_{2}$, $i=0, \ldots, k-3$. Then
$$C=B\backslash \{b_{0}\}-b_{1}:=\{c_{0}, c_{1}, \ldots, c_{k-4}, c_{k-3}\}. $$
Clearly,
$$c_{1}, c_{2}\in 2^{\wedge}C\backslash 2^{\wedge}(C\backslash \{c_{0}\}). $$
By $a_{4}=a_{1}+a_{3}$ and (\ref{eqa5.4}), we have
$$a_{5}=3a_{1}+a_{2}, \ \ a_{1}+2a_{2}, \ \ 2a_{1}+2a_{2} \text{ or } 3a_{1}+2a_{2}, $$
thus,
$$c_{3}=3a_{1}, \ \ a_{1}+a_{2}, \ \ 2a_{1}+a_{2} \text{ or } 3a_{1}+a_{2}. $$
If $c_{3}\in 2^{\wedge}(C\backslash \{c_{0}\})$, combining with $c_{1}=a_{1}$ and $c_{2}=2a_{1}$, then $c_{3}=3a_{1}$, thus, $b_{4}=2a_{1}+a_{2}$.
By (\ref{eqa5.14}) we have $b_{4}\in 2^{\wedge}(B\backslash \{b_{0}\})$, thus, $2a_{1}+a_{2}=2a_{2}-a_{1}$ or $2a_{2}$,
it follows that $a_{2}=3a_{1}$ or $2a_{1}$. By (\ref{eqa4.1}) we have $a_{1}=1$ and $a_{2}=2$ or $3$. By Fact A we have $|2^{\wedge}A|\geqslant 3k-7$.

Assume that $c_{3}\not\in 2^{\wedge}(C\backslash \{c_{0}\})$. Put
$$D=C\backslash \{c_{0}\}-c_{1}=A\backslash \{a_{0}, a_{1}, a_{2}\}-a_{3}. $$
If $\gcd(D)>1$, then by Proposition \ref{P12} (ii) and $\gcd(C)=1$ we have $|2^{\wedge}C|\geqslant 3k-12$, 
which contradicts with (\ref{eqa5.8}). Hence,
$$\gcd(D)=1. $$

Since $a_{k-1}-a_{3}\geqslant 2k-8$, we have
$$\max D=a_{k-1}-a_{3}\geqslant 2k-8=2|D|-2, $$
applying base case and basic fact to $D$, we have
$$|2^{\wedge}D|\geqslant 3k-16, $$
except for
$$D=\{0, a, b, a+b, 2a+b, a+2b, 2a+2b\}, $$
where $a$ and $b$ are two sufficiently large and coprime positive integers.

For the exceptional case, combining with (\ref{eqa5.4}) we have
$$A=\left\{\begin{array}{l}a_{0}, a_{1}, a_{2}, a_{1}+a_{2}, a+a_{1}+a_{2}, b+a_{1}+a_{2}, a+b+a_{1}+a_{2}, \\ 2a+b+a_{1}+a_{2}, a+2b+a_{1}+a_{2}, 2a+2b+a_{1}+a_{2}\end{array}\right\}. $$
By (\ref{eqa5.18}) we have $a_{4}=2a_{1}+a_{2}$, thus, $a=a_{1}$. Hence,
$$A=\left\{\begin{array}{l}a_{0}, a_{1}, a_{2}, a_{1}+a_{2}, 2a_{1}+a_{2}, b+a_{1}+a_{2}, b+2a_{1}+a_{2}, \\ b+3a_{1}+a_{2}, 2b+2a_{1}+a_{2}, 2b+3a_{1}+a_{2}\end{array}\right\}. $$
Again by (\ref{eqa5.4}) we have $a_{5}\in 2^{\wedge}(A\backslash\{a_{0}\})$, thus,
$$b\in\{2a_{1}, a_{2}, a_{1}+a_{2}, 2a_{1}+a_{2}\}. $$
By $a$ and $b$ are two sufficiently large and coprime positive integers, we have $b\neq 2a_{1}$.
Since $B=A\backslash\{a_{0}\}-a_{1}$, we have $b_{1}=a_{2}-a_{1}$, $b_{2}=a_{2}$, $b_{3}=a_{1}+a_{2}$, $b_{4}=b+a_{2}$.
Combining with (\ref{eqa5.14}) we have $b_{4}\in 2^{\wedge}(B\backslash \{b_{0}\})$, thus,
$$b\in \{a_{2}, a_{1}+a_{2}\}. $$

\begin{itemize}
  \item $b=a_{2}$. Then
  $$A=\left\{a_{0}, a_{1}, a_{2}, a_{1}+a_{2}, 2a_{1}+a_{2}, a_{1}+2a_{2}, 2a_{1}+2a_{2}, 3a_{1}+2a_{2}, 2a_{1}+3a_{2}, 3a_{1}+3a_{2}\right\}, $$
  thus,
  $$B=\left\{0, a_{2}-a_{1}, a_{2}, a_{1}+a_{2}, 2a_{2}, a_{1}+2a_{2}, 2a_{1}+2a_{2}, a_{1}+3a_{2}, 2a_{1}+3a_{2}\right\}. $$
  Combining with (\ref{eqa4.1}) we have $b_{6}\notin 2^{\wedge}(B\backslash \{b_{0}\})$, which contradicts with (\ref{eqa5.14}).
  \item $b=a_{1}+a_{2}$. Then
  $$A=\left\{a_{0}, a_{1}, a_{2}, a_{1}+a_{2}, 2a_{1}+a_{2}, 2a_{1}+2a_{2}, 3a_{1}+2a_{2}, 4a_{1}+2a_{2}, 4a_{1}+3a_{2}, 5a_{1}+3a_{2}\right\}, $$
  thus,
  $$B=\left\{0, a_{2}-a_{1}, a_{2}, a_{1}+a_{2}, a_{1}+2a_{2}, 2a_{1}+2a_{2}, 3a_{1}+2a_{2}, 3a_{1}+3a_{2}, 4a_{1}+3a_{2}\right\}. $$
  Similarly, we have $b_{5}\notin 2^{\wedge}(B\backslash \{b_{0}\})$, which contradicts with (\ref{eqa5.14}).
\end{itemize}

Hence, $|2^{\wedge}D|\geqslant 3k-16$. By $|2^{\wedge}C|=3k-13$, we have
$$c_{j}\in 2^{\wedge}(C\backslash \{c_{0}\}), \ \ j=4, \ldots, k-3. $$
If $a_{1}=1$ and $a_{2}<4$, then Fact A gives $|2^{\wedge}A|\geqslant 3k-7$.
Throughout the rest of the proof we may assume that $a_{1}\geqslant 2$ or $a_{1}=1$, $a_{2}\geqslant 4$.
By Proposition \ref{lem2.14}, it is sufficient to consider
$$A\cap [a_{1}, a_{6}]=\{2, 3, 5, 7, 8, 9\}, $$
thus,
$$B\cap [b_{1}, b_{5}]=\{1, 3, 5, 6, 7\}. $$
Since $b_{k-3}=a_{k-2}-a_{1}\geqslant 2k-6$, choose $t\in [6, k-3]$ such that $b_{t}\geqslant 2t$ and
$$b_{j}<2j, \ \ j=1, 2, \ldots, t-1. $$
Put $B_{1}=\{b_{0}, b_{1}, \ldots, b_{t}\}$. Then $|B_{1}|=t+1\geqslant 7$.
By Proposition \ref{P16} we have
\begin{equation}\label{P5-1}|2^{\wedge}B_{1}|\geqslant 3(t+1)-5=3t-2. \end{equation}

If $t\neq k-6$, then put $B_{2}=\{b_{t-2}, b_{t-1}, \ldots, b_{k-2}\}$.
Then
$$|B_{2}|=k-t+1\geqslant 4, \ \ |B_{2}|\neq 7. $$
Let $B^{-}_{2}=B_{2}-b_{t-2}$. Assume that $\gcd(B^{-}_{2})=d>1$. Then
$$b_{i}\equiv b_{t-2}\pmod{d}, \ \ i=t-2, \ldots, k-2.$$
Since $1, 3, 5, 6, 7\in B$, we have there exists
$$b_{j}\not\equiv b_{t-2}\pmod{d} \text{ for some } 1\leqslant j\leqslant t-3. $$
Let $b_{m}$ be the largest integer of $B$ satisfies $b_{m}\not \equiv b_{t-2}\pmod{d}$, then
$$b_{k-2}+b_{m}\notin 2^{\wedge}(B\backslash\{b_{k-2}\}), $$
it implies that
$$a_{k-1}+a_{m+1}\notin 2^{\wedge}(A\backslash\{a_{k-1}\}), $$
which contradicts with (\ref{eqa5.2}). Hence, $\gcd(B^{-}_{2})=1$.
By (\ref{eqa5.1}) and $b_{t-2}<2(t-2)$ we have
$$\max B^{-}_{2}=b_{k-2}-b_{t-2}\geqslant 2k-2t=2|B^{-}_{2}|-2, $$
applying basic fact to $B^{-}_{2}$, we have
\begin{equation}\label{P5-2}|2^{\wedge}B_{2}|=|2^{\wedge}B^{-}_{2}|\geqslant 3|B^{-}_{2}|-7=3k-3t-4. \end{equation}
In addition,
$$2^{\wedge}B_{1}\cap 2^{\wedge}B_{2}=\{b_{t-2}+b_{t-1}, b_{t-2}+b_{t}, b_{t-1}+b_{t}\}. $$
By (\ref{P5-1}) and (\ref{P5-2}) we have
$$|2^{\wedge}B|\geqslant |2^{\wedge}B_{1}|+|2^{\wedge}B_{2}|-3=3k-9. $$
Hence,
$$|2^{\wedge}A|\geqslant |2^{\wedge}(A\backslash\{a_{0}\})|+|\{a_{1}, a_{2}\}|=|2^{\wedge}B|+2=3k-7. $$

If $t=k-6$, then put $B_{2}=\{b_{t-3}, b_{t-2}, \ldots, b_{k-2}\}$. Then $|B_{2}|=8$.
Let $B^{-}_{2}=B_{2}-b_{t-3}$. Similarly, we have $\gcd(B^{-}_{2})=1$ and
$$|2^{\wedge}B_{2}|=|2^{\wedge}B^{-}_{2}|\geqslant 3|B^{-}_{2}|-7=17. $$
In addition,
$$2^{\wedge}B_{1}\cap 2^{\wedge}B_{2}=\{b_{i}+b_{j}: t-3\leqslant i<j\leqslant t\}. $$
Thus,
$$|2^{\wedge}B|\geqslant |2^{\wedge}B_{1}|+|2^{\wedge}B_{2}|-6=3k-9. $$
Hence,
$$|2^{\wedge}A|\geqslant |2^{\wedge}(A\backslash\{a_{0}\})|+|\{a_{1}, a_{2}\}|=|2^{\wedge}B|+2=3k-7. $$

\begin{ClaimD}Assume that $A$ satisfies
(\ref{eqa5.1}), (\ref{eqa5.2}), (\ref{eqa5.4}), (\ref{eqa5.23}), (\ref{eqa5.17}), (\ref{eqa5.18}), (\ref{eqa5.9}), (\ref{eqa5.7}), (\ref{eqa5.19}).
If $a_{k-1}-a_{3}<2k-8$, then $|2^{\wedge}A|\geqslant 3k-7$.
\end{ClaimD}

Since $a_{k-1}-a_{3}<2k-8$, we have $a_{k-4}^{*}<2k-8$.
It is clearly that $a_{k-3}^{*}\geqslant 2k-6$,
$$a_{k-5}^{*}+a_{k-1}^{*}=a_{k-4}^{*}+a_{k-2}^{*}, $$\vskip -5mm
\begin{equation}\label{eqa5.12-1}a_{k-5}^{*}+a_{k-2}^{*}\neq a_{k-4}^{*}+a_{k-3}^{*}\end{equation}
and for every $s\geqslant 5$,
\begin{equation}\label{eqa5.13-1}a_{k-1-s}^{*}+a_{k-2}^{*}=a_{k-1-u}^{*}+a_{k-1-v}^{*}  \text{ for some } 1<u<v<s. \end{equation}
By the above discussion, we need to consider $A^{*}$ satisfying
$$a_{k-1}^{*}-a_{2}^{*}\geqslant 2k-6, $$\vskip -5mm
\begin{equation}\label{eqa5.11}a_{4}^{*}=a_{1}^{*}+a_{3}^{*}, \end{equation} \vskip -5mm
\begin{equation}\label{eqa5.12}a_{4}^{*}\neq a_{3}^{*}+a_{2}^{*}-a_{1}^{*}\end{equation}
and for every $s\geqslant 5$,
\begin{equation}\label{eqa5.13}a_{s}^{*}=a_{u}^{*}+a_{v}^{*}-a_{1}^{*}  \text{ for some } 1<u<v<s. \end{equation}

If $a_{j}^{*}<2j$ for $j=1, \ldots, k-4$, then $a_{1}^{*}=1$ and $a_{2}^{*}<4$. By Fact B we have
$$|2^{\wedge}A|=|2^{\wedge}A^{*}|\geqslant3k-7. $$

If there exists an integer $s\in [2, k-4]$ such that
$$a_{s-1}^{*}\geqslant 2(s-1), \ \ a_{j}^{*}<2j, \ \ j=s, \ldots, k-4,$$
then
$2s-2\leqslant a_{s-1}^{*}<a_{s}^{*}<2s, $
thus,
\begin{equation}\label{eqa4.8-1}a_{s-1}^{*}=2s-2, \ \ a_{s}^{*}=2s-1. \end{equation}

If $s=k-4$, then $a_{k-5}^{*}=2k-10$, $a_{k-4}^{*}=2k-9$, thus, $a_{4}-a_{3}=1$.
By \eqref{eqa5.18} we have $a_{1}=1$, thus, $a_{3}=a_{2}+1$, combining with \eqref{eqa5.17}, we have $a_{k-1}-a_{3}=a_{k-1}-a_{2}-1\geqslant 2k-7$, a contradiction.

If $s\leqslant k-5$, then $a_{s+1}^{*}<2s+2$. Combining with (\ref{eqa4.8-1}), we have
\begin{equation}\label{eqa4.10-1}a_{s+1}^{*}=2s \text{ or } 2s+1. \end{equation}
Write
$$\widetilde{A}_{1}=\{a_{0}^{*}, a_{1}^{*}, \ldots, a_{s-1}^{*}, a_{s}^{*}, a_{s+1}^{*}\}, $$\vskip-5mm
$$\widetilde{A}_{2}=\{a_{s-1}^{*}, a_{s}^{*}, a_{s+1}^{*}, \ldots, a_{k-1}^{*}\}. $$
Define the set $\widetilde{A}^{-}_{2}$ by
$$\widetilde{A}^{-}_{2}=\widetilde{A}_{2}-a_{s-1}^{*}=\{0, 1, a_{s+1}^{*}-a_{s-1}^{*}, \ldots, a_{k-1}^{*}-a_{s-1}^{*}\}. $$
Then
\begin{equation}\label{eqa5.6-1}\max\widetilde{A}^{-}_{2}=a_{k-1}^{*}-a_{s-1}^{*}\geqslant 2(k-s+1)-2=2|\widetilde{A}^{-}_{2}|-2. \end{equation}
Noting that the third smallest integer of $\widetilde{A}^{-}_{2}$ is either 2 or 3.
If $|\widetilde{A}^{-}_{2}|=7$ and $|2^{\wedge}\widetilde{A}^{-}_{2}|<14$, then by Base Case we have the largest integer of $\widetilde{A}^{-}_{2}$ is not greater than 8, which contradicts with (\ref{eqa5.6-1}).
Hence, if $|\widetilde{A}^{-}_{2}|=7$, then
\begin{equation}\label{e2-2-2}|2^{\wedge}\widetilde{A}^{-}_{2}|\geqslant 14. \end{equation}
Assume that $|\widetilde{A}^{-}_{2}|\neq 7$. By (\ref{eqa5.6-1}), applying basic fact to $\widetilde{A}^{-}_{2}$, we have
\begin{equation}\label{e2-2-1}|2^{\wedge}\widetilde{A}_{2}|=|2^{\wedge}\widetilde{A}^{-}_{2}|\geqslant 3k-3s-4. \end{equation}

If $|2^{\wedge}\widetilde{A}_{1}|\geqslant 3|\widetilde{A}_{1}|-6=3s$, then by
$$2^{\wedge}\widetilde{A}_{1}\cup2^{\wedge}\widetilde{A}_{2}\subseteq 2^{\wedge}A^{*}$$
and
$$2^{\wedge}\widetilde{A}_{1}\cap2^{\wedge}\widetilde{A}_{2}=\{a_{s-1}^{*}+a_{s}^{*}, a_{s-1}^{*}+a_{s+1}^{*}, a_{s}^{*}+a_{s+1}^{*}\}, $$
combine with (\ref{e2-2-1}), (\ref{e2-2-2}) we have
$$|2^{\wedge}A|=|2^{\wedge}A^{*}|\geqslant |2^{\wedge}\widetilde{A}_{1}|+|2^{\wedge}\widetilde{A}_{2}|-3\geqslant 3k-7. $$

If $|2^{\wedge}\widetilde{A}_{1}|=3|\widetilde{A}_{1}|-7=3s-1$, then by (\ref{factd}), (\ref{eqa4.8-1}), (\ref{eqa4.10-1}), Proposition \ref{P7-2} and (\ref{eqa5.11})-(\ref{eqa5.13}) we need consider the following cases:

\begin{itemize}
  \item $\widetilde{A}_{1}=\{0, 2, 3, 5\}$. Then $s=2$. Since
  $$a_{j}^{*}<2j, \ \ j=2, \ldots, k-4, $$
  combining with (\ref{eqa5.23}) and the fact that $a_{1}^{*}=2$, we have
  $$a_{j}^{*}-a_{1}^{*}<2(j-1), \ \ j=2, \ldots, k-4; \ \ a_{k-3}^{*}-a_{1}^{*}\geqslant 2k-8. $$
  Moreover, for all $t=1, \ldots, k-6$, by (\ref{eqa5.13-1}) we have
  $$(a_{k-2}^{*}-a_{1}^{*})+(a_{t}^{*}-a_{1}^{*})\in 2^{\wedge}(\widetilde{A}^{-}_{2}\backslash\{a_{k-2}^{*}-a_{1}^{*}, a_{k-1}^{*}-a_{1}^{*}\}). $$
  By (\ref{factd}) we have $a_{4}^{*}\geqslant 7$, thus, $a_{4}^{*}=7$.
  Hence, $1, 3, 5\in \widetilde{A}^{-}_{2}\backslash\{a_{k-2}^{*}-a_{1}^{*}, a_{k-1}^{*}-a_{1}^{*}\}$.
  By Proposition \ref{P16} and Proposition \ref{P16.1} we have
  $$|2^{\wedge}(\widetilde{A}_{2}\backslash\{a_{k-2}^{*}, a_{k-1}^{*}\})|=|2^{\wedge}(\widetilde{A}^{-}_{2}\backslash\{a_{k-2}^{*}-a_{1}^{*}, a_{k-1}^{*}-a_{1}^{*}\})|\geqslant 3(k-3)-5=3k-14. $$
  By (\ref{eqa5.12-1}) we have
  $$|2^{\wedge}(\widetilde{A}_{2}\backslash\{a_{k-1}^{*}\})|\geqslant 3k-11, $$
  thus,
  $$|2^{\wedge}\widetilde{A}_{2}|\geqslant |2^{\wedge}(\widetilde{A}_{2}\backslash\{a_{k-1}^{*}\})|+|\{a_{k-3}^{*}+a_{k-1}^{*}, a_{k-2}^{*}+a_{k-1}^{*}\}|\geqslant 3k-9, $$
  it follows that
  $$|2^{\wedge}A|=|2^{\wedge}A^{*}|\geqslant |2^{\wedge}\widetilde{A}_{2}|+|\{a_{1}^{*}, a_{2}^{*}\}|\geqslant 3k-7. $$
  \item $\widetilde{A}_{1}=\{0, 1, 4, 5, 6\}$. Then $s=3$. Thus,
  $$a_{k-3}^{*}-a_{2}^{*}\geqslant 2k-10, \ \ a_{j}^{*}-a_{2}^{*}<2(j-2), \ \ j=3, \ldots, k-4. $$
  By $0, 1, 2\in \widetilde{A}^{-}_{2}\backslash\{a_{k-2}^{*}-a_{2}^{*}, a_{k-1}^{*}-a_{2}^{*}\}$ and Theorem F we have
  $$|2^{\wedge}(\widetilde{A}_{2}\backslash\{a_{k-2}^{*}, a_{k-1}^{*}\})|=|2^{\wedge}(\widetilde{A}^{-}_{2}\backslash\{a_{k-2}^{*}-a_{2}^{*}, a_{k-1}^{*}-a_{2}^{*}\})|\geqslant 3k-18, $$
  thus,
  $$|2^{\wedge}\widetilde{A}_{2}|\geqslant 3k-13. $$
  Noting that
  $$\{1, 4, 5, 6, 7\}\subseteq 2^{\wedge}A^{*}\backslash 2^{\wedge}\widetilde{A}_{2}, $$
  it follows that
  $$|2^{\wedge}A^{*}\backslash 2^{\wedge}\widetilde{A}_{2}|\geqslant 5. $$
  If $|2^{\wedge}A^{*}\backslash 2^{\wedge}\widetilde{A}_{2}|\geqslant 6$, then
  $$|2^{\wedge}A|=|2^{\wedge}A^{*}|\geqslant 3k-7. $$
  Assume that $|2^{\wedge}A^{*}\backslash 2^{\wedge}\widetilde{A}_{2}|=5$.
  If $a_{5}^{*}=7$ or $8$, then $8\in 2^{\wedge}A^{*}\backslash 2^{\wedge}\widetilde{A}_{2}$, a contradiction.
  Since $a_{5}^{*}<10$, we have $a_{5}^{*}=9$. Similarly, we have $a_{6}^{*}=10$, $a_{7}^{*}=13. $
  Thus,
  $$1, 2, 5, 6, 9\in \widetilde{A}^{-}_{2}\backslash\{a_{k-2}^{*}-a_{2}^{*}, a_{k-1}^{*}-a_{2}^{*}\}. $$
  By (\ref{eqa5.13-1}) we have
  $$(a_{k-2}^{*}-a_{2}^{*})+(a_{j}^{*}-a_{2}^{*})\in 2^{\wedge}(\widetilde{A}^{-}_{2}\backslash\{a_{k-2}^{*}-a_{2}^{*}, a_{k-1}^{*}-a_{2}^{*}\}) $$
  for all $j=2, \ldots, k-6$.
  By Proposition \ref{P16.3} we have
  \begin{eqnarray*}|2^{\wedge}(\widetilde{A}_{2}\backslash\{a_{k-2}^{*}, a_{k-1}^{*}\})|&=&|2^{\wedge}(\widetilde{A}^{-}_{2}\backslash\{a_{k-2}^{*}-a_{2}^{*}, a_{k-1}^{*}-a_{2}^{*}\})|\\
  &\geqslant& 3(k-4)-5=3k-17.
  \end{eqnarray*}
  Combining with (\ref{eqa5.12-1}) we have
  $$|2^{\wedge}(\widetilde{A}_{2}\backslash\{a_{k-1}^{*}\})|\geqslant 3k-14, $$
  thus,
  \begin{eqnarray*}|2^{\wedge}A|=|2^{\wedge}A^{*}|&\geqslant& 5+|2^{\wedge}(\widetilde{A}_{2}\backslash\{a_{k-1}^{*}\})|+|\{a_{k-3}^{*}+a_{k-1}^{*}, a_{k-2}^{*}+a_{k-1}^{*}\}|\\
  &\geqslant& 3k-7. \end{eqnarray*}
\end{itemize}

By Claims 1-4, we complete the proof of Theorem \ref{T}.



\section*{References}

\begin{thebibliography}{30}
\bibitem{Freiman64}G.A. Freiman, {\it Addition of finite sets}, Doklady Akad. Nauk SSSR, 158(1964), 1038-1041.

\bibitem{Freiman1} G.A. Freiman, {\it Foundations of structural theory of set addition}, vol. 37, Translations of Mathematical Monographs, American Mathematical Society, Providence, R. I. 1973.

\bibitem{Freiman} G.A. Freiman, {\it The addition of finite sets}, I, Izv. Vys\v{s}. U\v{c}ebn. Zaved. Matematika 13(1959), 202-213.

\bibitem{Freiman1999} G.A. Freiman, L. Low and J. Pitman, {\it Sumsets with distinct summands and the conjecture of Erd\H{o}s-Heilbronn on sums residues}, Asterisque 258(1999), 163-172.

\bibitem{Jin2007} R.L. Jin, {\it Freiman's inverse problem with small doubling property}, Adv. Math. 216(2007), 711-752.

\bibitem{Lev2} V.F. Lev, {\it Restricted set addition in groups, I. The classical setting}, J. London Math. Soc. 62(2000), 27-40.

\bibitem{Nathanson1} M.B. Nathanson, {\it Inverse theorems for subset sums}, Trans. Amer. Math. Soc. 347(1995), 1409-1418.

\bibitem{Sch} T. Schoen, {\it The cardinality of restricted sumsets}, J. Number Theory 96(2002), 48-54.

\bibitem{Wang2025}Y.J. Wang and M. Tang, {\it Freiman-Lev conjecture (in Chinese)}, Sci Sin Math. 55(2025), 1-32, doi:10.1360/SSM-2024-0249. (also see arXiv: 2402.01471v3)

\bibitem{Wang} Y.J. Wang and M. Tang, {\it On Freiman-Lev conjecture}, Contrib. Discrete Math. 20(2025), 42-59.

\end{thebibliography}
\end{document}